\documentclass[12pt,reqno]{amsart}
\usepackage{amsthm}
\usepackage{amssymb}
\usepackage{amsmath}
\usepackage{mathrsfs}
\usepackage{enumitem}
\usepackage{amsfonts}
\usepackage{amssymb,amsmath}
\usepackage[colorlinks, linkcolor=black, citecolor=blue,   hypertexnames=false]{hyperref}
\usepackage[numbers,sort&compress]{natbib}
\newcommand{\bX}{\mathbf{X}}

\newcommand{\bH}{\mathbf{H}}
\newcommand{\bh}{\mathbf{h}}

\newcommand{\R}{{\mathbb{R}}}
\newcommand{\Z}{{\mathbb Z}}

\newcommand{\N}{{\mathbb N}}

\newcommand{\FF}{{\mathcal F}}

\newcommand{\HH}{{\mathcal H}}

\newcommand{\RR}{{\mathcal R}}
\renewcommand{\SS}{{\mathcal S}}

\newcommand{\tri}{|\!|\!|}

\newcommand{\Tr}{\operatorname{Tr}}

\newcommand{\tr}{\operatorname{tr}}

\renewcommand{\Re}{\mathop{\rm Re}\nolimits}
\renewcommand{\Im}{\mathop{\rm Im}\nolimits}

\DeclareMathOperator{\Ric}{Ric}

\theoremstyle{plain}

\newtheorem{thm}{Theorem}[section]
\newtheorem{prop}[thm]{Proposition}
\newtheorem{cor}[thm]{Corollary}
\newtheorem{lemma}[thm]{Lemma}

\theoremstyle{definition}

\newtheorem{rem}{Remark}[thm]
\newtheorem{defn}[thm]{Definition}

\numberwithin{equation}{section}
\newcommand{\thmref}[1]{Theorem~\ref{#1}}
\newcommand{\secref}[1]{Section~\ref{#1}}

\newcommand{\propref}[1]{Proposition~\ref{#1}}

\def\squarebox#1{\hbox to #1{\hfill\vbox to #1{\vfill}}}

\newcommand{\<}{\langle}
\renewcommand{\>}{\rangle}

\renewcommand{\d}{\partial}
\newcommand{\ep}{\epsilon}
\newcommand{\lV}{\lVert}
\newcommand{\rV}{\rVert}

\def\be{{\beta}}
\def\ga{\gamma}
\def\Ga{\Gamma}
\def\de{\delta}
\def\De{\Delta}
\def\ep{\epsilon}

\def\la{\lambda}

\def\si{\sigma}
\def\Si{\Sigma}

\def\nab{\nabla}

\def\al{\alpha}

\usepackage{xcolor}

\title[Skew Mean Curvature Flow]
{Global regularity of Skew mean curvature flow for small data in $d\geq 4$ dimensions}

\author[J. Huang]
{Jiaxi Huang}

\author[Z. Li]
{Ze Li}

\author[D. Tataru]
{Daniel Tataru}

\address{School of Mathematics and Statistics, Beijing Institute of Technology, 
\newline\indent
Beijing
100081, P.R. China}
\email{jiaxih@bit.edu.cn}

\address{\newline\indent
School of Mathematics and Statistics, Ningbo University,
\newline\indent
Ningbo, 315211, P.R. China
}

\email{rikudosennin@163.com}

\address{\newline\indent
Department of Mathematics, University of California at Berkeley,
\newline\indent
Berkeley, CA 94720, USA
}

\email{tataru@math.berkeley.edu}



\begin{document}

\begin{abstract}
The skew mean curvature flow is an evolution equation for a $d$ dimensional manifold immersed into $\mathbb{R}^{d+2}$, and which moves along the binormal direction with a speed 
proportional to its mean curvature.

In this article, we prove small data global regularity in low-regularity Sobolev spaces for the skew mean curvature flow in dimensions $d\geq 4$. This extends the local 
well-posedness result in \cite{HT}.

\end{abstract}

\date{\today}
\maketitle


\section{Introduction}

The skew mean curvature flow (SMCF) evolves a codimension 2 submanifold along its binormal direction with a speed given by its
mean curvature. Precisely speaking, assume that $\Sigma$  is a  $d$-dimensional oriented manifold and $(\mathcal{M},g_{\mathcal M})$ is an  $(d+2)$-dimensional oriented Riemannian manifold, then
SMCF is  a family of time-dependent immersions $F:\mathbb{I}\times \Sigma \to \mathcal{M}$  satisfying
\begin{align}            \label{1}
\left\{
  \begin{aligned}
     &\ {\partial_t}F= J(F)\mathbf{H}(F), \qquad (t,x)\in \mathbb I\times \Sigma, \hbox{ } \\
    &\ F(0,x)=F_0(x),  \hbox{}
  \end{aligned}
\right.
\end{align}
where, for each given $t\in {\mathbb I}$, $\mathbf{H}(F)$ denotes the mean curvature vector
of the submanifold $\Sigma_t:= F(t,\Sigma)$. Here $J(F)$, which denotes the  natural induced complex structure for the normal bundle $N\Sigma_t$,  can be simply defined as rotating a vector in the normal space by $\frac{\pi}{2}$ positively (notice that $N\Sigma_t$ is of rank 2).
An alternative formulation  of SMCF is
\begin{equation}           \label{Main-Sys}
\left\{
  \begin{aligned}
&\ (\partial_t F)^{\perp}=J(F)\mathbf{H}(F),\\
&\ F(0,\cdot)=F_0.
 \end{aligned}
\right.
\end{equation}
Here, for an arbitrary vector $Z\in T\mathcal{M}$ at $F$,    $Z^\perp$
denotes its orthogonal projection onto $N\Sigma_t$. Note that (\ref{1}) differs from (\ref{Main-Sys}) by a time dependent diffeomorphism of $\Sigma_t$. Hence, (\ref{1}) and  (\ref{Main-Sys}) are topologically equivalent, 
but \eqref{Main-Sys} has a larger gauge group consisting 
of all space-time changes of coordinates.

For $d=1$, the 1-dimensional SMCF in  $\mathbb R^3$ is the vortex filament equation
$
\partial_t v= \partial_sv\times \partial^2_sv
$
for $v:(s,t)\in \mathbb R\times \mathbb R\longmapsto v(s,t)\in \mathbb R^3$, where $t$ denotes time, $s$ denotes the arc-length parameter of the curve $v(t,\cdot)$,  and $\times$ denotes
the cross product in $\mathbb R^3$. The  vortex filament equation describes  the free motion
of a vortex filament, see Da Rios \cite{R}, Hasimoto \cite{Ha}.
For $d\ge 2$, the (SMCF) was deduced by both physicists and  mathematicians. The physical  motivations  are   the localized
induction approximation (LIA) of high dimensional  Euler equations and  asymptotic dynamics of vortices in
superconductivity and superfluidity, see Lin \cite{Lin}, Jerrard \cite{J}, Shashikanth \cite{S}, Khesin \cite{K}.
SMCF  also appears in various mathematical problems, especially  the Hamiltonian flow associated with Marsden-Weinstein sympletic structure \cite{MW}, nonlinear Grassmannian manifolds discussed by Haller-Vizman \cite{HaVi04}, and the star mean curvature flow introduced by Terng \cite{T}.
Moreover, it is remarkable that SMCF has a deep relationship with the Schr\"odinger map flow (e.g. \cite{UT}), in fact,
\cite{S1} proved that the Gauss map of a $d$-dimensional SMCF in $\mathbb R^{d+2}$ satisfies a Schr\"odinger map flow equation.

Let us briefly recall some earlier works on SMCF.
The $1$-d case is special, in that the problem has a semilinear, rather than quasilinear structure, and is essentially equivalent to the $1$-d cubic NLS problem.
For more details  we refer the reader to the survey article of Vega~\cite{Ve14}.

The early work of Gomez \cite{Gomez} proposed a way to write SMCF as a quasilinear Schr\"odinger equation system by introducing a complex valued  scalar mean curvature and choosing some gauge for the normal boundle. The model (\ref{1}) was studied by Song-Sun \cite{SS}, who considered the local Cauchy problem   in dimension $d=2$
proved local existence of SMCF for the for $F:\Sigma \to \mathbb R^{4}$ with a compact oriented surface $\Sigma$. This was generalized by Song \cite{S2} to $F:\Sigma \to \mathbb R^{d+2}$ with a compact oriented manifold $\Sigma$ for all $d\ge 2$, and \cite{S2} proved existence and uniqueness of smooth solutions for an  arbitrary oriented manifold $\Sigma$. However, as noted also in \cite{S2}, when attempting to study the SMCF in Sobolev spaces using the formulation in  \eqref{1} there is a one derivative loss, which indicates that this might not be the best way to choose the space-time coordinates. 

The above issue was clarified in   Huang-Tataru \cite{HT,HT22}, who proposed an alternative approach, namely to start  with the formulation 
of SMCF in  (\ref{Main-Sys}), and then to choose a 
favourable space-time gauge (i.e. coordinates).
In this gauge there is no more loss of derivatives,
and  they were able to prove a full local well-posedness result in low regularity Sobolev spaces for initial  data which are small perturbations of flat metrics.
Precisely, the solutions obtained in \cite{HT,HT22} are 
at regularity $H^s$, with $s > \dfrac{d}{2}$,
measured at the curvature level; this is one derivative above scaling.
The gauge formulation of the SMCF flow 
in \cite{HT,HT22} closely resembles a quasilinear 
Schr\"odinger equation, coupled with several elliptic/parabolic equations. 
For the local well-posedness theory of general  quasilinear Schr\"odinger equations, see the pioneer works of Kenig-Ponce-Vega \cite{KPV2,KPV3,KPV} for localized initial data, as well as Marzuola-Metcalfe-Tataru~\cite{MMT1,MMT1,MMT3} for data in translation invariant $H^s$ based spaces.

The small data global regularity  problem for SMCF  in the  formulation (\ref{1}) was  considered 
in \cite{Li1}  which proved that Euclidean planes are stable under SMCF for small transversal perturbations in some $W^{2,q}\cap H^k$ space  with some $q\in(1,2)$.
In the later work \cite{Li2},   the $W^{2,q}$ smallness and transversal assumption of \cite{Li1} were removed in $d\ge 3$, and it proved the global in time time existence and  of scattering of small data solutions and the existence of wave operators.

\subsection{The main result}

Our objective in this paper is to establish the global 
in time well-posedness and scattering for solutions to 
SMCF in the formulation (\ref {Main-Sys}) for small initial data.  

Our main dynamic variable will be the \emph{complex mean curvature} $\psi$ for our system, which is defined in the next section, see \eqref{csmc}, and stands for the representation of the scalar mean curvature
relative to an orthonormal frame in $N\Sigma$ determined by our choice of gauge. The similar representation of the 
full second fundamental form will be denoted by $\lambda$,
and the two are related by $\psi = \Tr \lambda$.

 To measure the Sobolev regularity of $\psi$ for our  global solutions  we introduce the index $s_d$ so that 
 \begin{align}     \label{Reg-index}
     s_d\geq 3,\quad \text{if } d=4;\qquad s_d>\frac{d+1}{2}+\frac{1}{2(d-1)},\quad \text{if } d\geq 5,
 \end{align}
To measure the averaged (Strichartz) decay of the
solutions in time we will use the exponent $r_d$ defined by 
\begin{align}     \label{r_d}
     r_d=\frac{2d(d-1)}{(d-2)^2},\quad \text{for } d\geq 4.
\end{align}
Then we define the Strichartz norms $S[0,T]$ as 
\begin{align}\label{S4}
	\|\psi\|_{S[0,T]}:=\|\psi\|_{L^2(0,T;W^{1,4})}+\|\psi\|_{L^2(0,T; W^{s_d-2,r_d})},\quad  \text{for }d=4,
\end{align}
and 
\begin{align} \label{S5}
	\|\psi\|_{S[0,T]}:=\|\psi\|_{L^2(0,T; W^{s_d-2,r_d})},\quad  \text{for }d\geq 5.
\end{align}

At this point, we content ourselves with a less precise formulation of the main result, relative to the harmonic/Coulomb gauge which was introduced in \cite{HT} and is discussed in Section~\ref{gauge}:
\begin{thm}[Small data global regularity and scattering]   \label{GWP-thm}
Let $s_d$ and $r_d$ be as \eqref{Reg-index}, \eqref{r_d} respectively for $d\geq 4$. Then there exists $\ep_0>0$ sufficiently small such that, for all initial data $\Sigma_0$ with metric and mean curvature satisfying
 \begin{align*}
\|\d_x(g_0-I)\|_{H^{s_d}}+\|{\bf H}_0\|_{H^{s_d}} \leq \ep_0,
 \end{align*}
 the skew mean curvature flow \eqref{Main-Sys} for maps from $\R^d$ to the Euclidean space $(\R^{d+2},g_{\R^{d+2}})$ is globally well-posed in the harmonic/Coulomb gauge. 
 
 Moreover, in the harmonic/Coulomb gauge, the metric and complex mean curvature satisfy the bounds 
 \begin{equation}\label{psi-full-reg}
	\|\d_x (g-I_d)\|_{C_t H^{s_d+1}_x}+\lV \psi\rV_{S(\R)}+\lV \psi\rV_{C _tH^{s_d}_x} \lesssim \lV \psi_0\rV_{H^{s_d}_x}.
	\end{equation}
    In addition, there exists $\psi_{\pm}\in H^{s_d-2} $ such that
    \begin{equation}   \label{Scatter}
    \lim_{t\rightarrow\pm\infty} \|\psi-e^{it\De}\psi_{\pm}\|_{H^{s_d-2}_x}=0.
    \end{equation}
\end{thm}

\begin{rem}
The gauge choice used for the above result is 
the \emph{harmonic/Coulomb gauge}, following \cite{HT}.
Here harmonic refers to the choice of coordinates on 
$\Sigma$ at fixed time, and Coulomb applies to the 
choice of the orthonormal frame on $N \Sigma$. In this gauge, the surface $\Sigma$ is uniquely determined up to symmetries by the complex mean curvature $\psi$ at fixed time in an elliptic fashion.  By contrast, in \cite{HT22} the harmonic/Coulomb gauge is only imposed at the initial time, while a heat gauge is used forward in time.

\end{rem}

\begin{rem}\label{Cho-Reg}
One may compare the Sobolev index $s_d$ in the theorem 
with the weaker restriction $s > d/2$ in \cite{HT,HT22}.
Here the choice of regularity $s_d$ is more restricted 
due to the need to also control decay via global in time Strichartz norms. Precisely, our main control norm for 
the energy estimates will essentially be $\|\lambda\|_{L^2_t L^\infty_x}$, 
see \eqref{E-psi} below. To bound this by $\|\psi\|_{L^2_TW^{s_d-2,r_d}}$ by elliptic estimates and Sobolev embeddings requires that  $s_d>\frac{d}{r_d}+2$. This gives the $s_d$ threshold \eqref{Reg-index} for $d\geq 5$. 

	In dimension $d=4$ we face an additional obstruction
	arising in the study of the global well-posedness for the linearized equation. For that we need  Strichartz estimates in the space $L^2W^{1,4}$, which in turn restricts the regularity to $s_d\geq 3$.
\end{rem}

The global regularity is closely related to the energy estimates and Strichartz estimates for the complex mean curvature $\psi$ for our system. Following \cite{HT}, 
in the harmonic/Coulomb gauge $\psi$ 
solves a quasilinear Schr\"odinger equation
\eqref{mdf-Shr-sys-2}-\eqref{ell-syst}.
We describe these estimates next,
beginning with the energy estimates.

A key point in the following proposition is that
we should work with the ``good" energy, which is both coercive and propagates well along the flow. At integer
Sobolev regularity indices there is a canonical, geometric  choice, given by the $L^2$ norm of covariant derivatives
of $\psi$. The challenge is then to prove coercivity,
which is no longer a covariant property but depends instead on our gauge choice.

\begin{prop}[Energy estimates in $H^k$]         \label{En-prop}
For each a nonnegative integer $k$ there exists 
an energy functional $E^k=E^k(\psi)$ defined on 
functions in $H^k$ which are also small in $H^s$ for some 
$s > d/2$, with has the following two properties:

\begin{enumerate}[label=\roman*)]
\item\ [Coercivity:] In the harmonic/Coulomb gauge we have the equivalence relation:
\begin{equation}      \label{eq-E}
    E^k(\psi) \approx_{C_1} \| \psi\|_{H^k}^2\approx \| \la\|_{H^k}^2 \approx \| \la\|_{\mathsf H^k}^2,
\end{equation}
where the constant $C_1$ only depends on the $H^s$ norm of $\psi$.

\item\ [Energy growth] 
If $\psi$ is a solution of the SMCF flow \eqref{mdf-Shr-sys-2}-\eqref{ell-syst} 
which remains small in $H^s$ then 
\begin{equation}  \label{Energy}
   \frac{d}{dt} E^k(\psi) \leq  C_E\| \lambda \|_{L^\infty}^2 \|\la\|_{\mathsf H^k}^2.
 \end{equation}  
 \end{enumerate}	
\end{prop}

As a consequence of \eqref{Energy}, by \eqref{eq-E} and Gronwall's inequality we obtain
\begin{align}     \label{E-psi}
	\| \psi(t)\|_{H^{k}}^2\leq  C_1^2  e^{\int_0^T C_E\|\la\|_{L^\infty}^2 ds}\|\psi_0\|_{H^k}^2.
\end{align}
This justifies the need to control the norm 
$\|\la\|_{L^2 L^\infty}$ for our global solutions.

\begin{rem}
    The energy estimate \eqref{Energy} holds without any gauge assumptions, and was proved first in \cite[Lemma 4.9]{SS}. Here we use a different method to prove this estimate, using only the Schr\"odinger equation for $\psi$ and the associated constraints to gain the estimates. The gauge choice is, however, essential for the coercivity part.
\end{rem}

\begin{rem}
     The energies are constructed in an explicit fashion only for integer $k$. Nevertheless, as a consequence in our analysis in the last section of the paper, it follows that
     bounds of the form \eqref{E-psi} hold also for all noninteger $k>0$. However, we do this using
     a mechanism which is akin to a paradifferential expansion, without constructing an explicit
     energy functional as provided by the above theorem in the integer case.
\end{rem}

We now turn our attention to the Strichartz estimates for $\psi$. Since our problem is quasilinear, here we a-priori assume that $\psi$ remains small in $H^{s_d}$, and 
we also loose some derivatives.

\begin{prop}[Strichartz estimates]   \label{impbd-1}
	Let $s_d$ be as \eqref{Reg-index} and $\si_d=s_d-2$ for $d\geq 4$. Assume that $\psi$ is a solution of \eqref{mdf-Shr-sys-2}-\eqref{ell-syst} on some interval $[0,T]$ for $T>1$,  which satisfies the smallness condition
	\begin{equation*} 
	\lV \psi\rV_{L^\infty H^{s_d}}\leq C_0\ep_0,
    \end{equation*}
Then $\psi$ satisfies  the Strichartz bound
\begin{equation}  \label{psi-L2t}
	\|\psi(t)\|_{S[0,T]}\leq C_2(\|\psi_0\|_{H^{\si_d+\frac{d-2}{2(d-1)}}}+(C_0\ep_0)^2\|\psi(t)\|_{S[0,T]}).
\end{equation}
\end{prop}

A starting point for this result is provided by the endpoint Strichartz estimates of Keel-Tao \cite{KT}.
However, in addition we also use the larger class
of inhomogeneous Strichartz estimates developed \cite{F05, Vilela, Koh, KohSeo}. The latter play a key role in lowering the regularity assumptions for the initial data in our theorem.

\subsection{An outline of the proof} 
There are several key steps in the proof of our main result:

\medskip
\emph{1. The gauge choice.}
The formulation (\ref {Main-Sys}) has a key additional gauge freedom compared with the equation (\ref{1}). Indeed, (\ref {Main-Sys})  is invariant under any time dependent diffeomorphism in $\Sigma_t$, while  (\ref{1}) is only invariant under time independent  diffeomorphisms in $\Sigma_t$. This additional freedom enabled us  \cite{HT} to use the harmonic coordinate system.  This is then combined with the the Coulomb gauge for the orthonormal frame on the normal bundle. This reformulation of the equation \eqref{Main-Sys} is reviewed in Section~\ref{gauge}, where
we rewrite it as a nonlinear Schr\"odinger equation for a single independent
variable. This independent variable, denoted by $\psi$, represents the trace of the second fundamental form on $\Sigma_t$, in complex notation. In addition to the independent variables, we will use several dependent variables, as follows:
\begin{itemize}
    \item The Riemannian metric $g$ on $\Sigma_t$.
    \item The (complex) second fundamental form $\lambda$ for $\Sigma_t$.
    \item The magnetic potential $A$, associated to the natural connection on the
    normal bundle $N \Sigma_t$, and the corresponding temporal component $B$.
    \item The advection vector field $V$, associated to the time dependence of our choice of coordinates.
\end{itemize}
These additional variables will be viewed as uniquely determined by our independent variable $\psi$, provided that a suitable gauge choice was made; in our case this gauge is the combined harmonic/Coulomb gauge.
Thus  (\ref {Main-Sys}) reduces to  
\begin{enumerate}
	\item[(a)] A nonlinear Schr\"odinger equation for $\psi$, see \eqref{mdf-Shr-sys-2};
	\item[(b)] An elliptic fixed time system \eqref{ell-syst} for the dependent variables
	$\SS=(g,\lambda,V,A,B)$, together with suitable compatibility conditions (constraints).
\end{enumerate}
At the conclusion of Section~\ref{gauge} we provide a gauge version
of our main result, see Theorem~\ref{GWP-MSS-thm}.

\medskip

\emph{2. Elliptic estimates.} In Section~\ref{Ellip-sec}, we then consider the space-time bounds for the elliptic system \eqref{ell-syst} and the associated linearized equations.  Such bounds have already been proved in \cite{HT} at the level of the $H^s$ spaces. But here we also need similar bounds at the level of the Strichartz norms, which capture the time decay  of $\la$ and $\SS$ in terms  of the corresponding decay bounds for $\psi$. 
Another novelty here is that we also prove elliptic bounds for the
linearized system with $\psi_{lin}$ in $H^{-1}$; this is in contrast to 
\cite{HT}, where only nonnegative Sobolev norms were used.
\medskip

\emph{3. Energy estimates.}
In \secref{Sec-Energy}, we turn our attention to the energy estimates in \propref{En-prop}. Here we use the intrinsic Sobolev spaces $\mathsf H^k$ to define energy functional, and give the related energy estimates. We also prove an energy estimate for the linearized Schr\"odinger equation, which will be  needed 
in particular in order to transfer energy bounds from integer to fractional Sobolev spaces. 

\medskip

\emph{4. Strichartz estimates.}
The Strichartz estimates for $\psi$  are proved in \secref{Sec-Strichartz} 
using the Schr\"odinger system \eqref{mdf-Shr-sys-2}. Since this is a quasilinear 
problem, we cannot directly work with the linear variable coefficient system. 
Instead, we prove Proposition~\ref{impbd-1} using a bootstrap argument which is based  on the Strichartz estimates for the flat Schr\"odinger evolution, namely  Keel-Tao's endpoint Strichartz estimates and inhomogeneous Strichartz estimates, see \cite{F05,Koh,KohSeo}.

\medskip

\emph{ 5. The final bootstrap.} In the last section of the paper, we gain the $H^s$ solutions as a limit of solutions in higher order Sobolev spaces. Using the energy estimates in $\mathsf H^N$ and the energy estimates of linearized equation, we prove the improved energy bounds for $\psi$ in fractional Sobolev spaces. This in turn allows us to close the high level bootstrap loop for both the energy estimates and the Strichartz estimates, as stated in  \propref{bootstrap-prop}.
As a byproduct, we also obtain the scattering result Schr\"odinger equation for $\psi$ in the weaker Sobolev norms $H^{s_d-2}$.

\bigskip

\section{The differentiated equations and the gauge choice}
\label{gauge}

The goal of this section is to introduce the main independent variable $\psi$, which represents
the trace of the second fundamental form in complex notation, as well as the
following auxiliary variables: the metric $g$, the second fundamental form $\lambda$,
the connection coefficients $A,B$ for the  normal bundle as well as the advection vector field $V$.
For $\psi$ we start with \eqref{Main-Sys} and derive a nonlinear Sch\"odinger type system \eqref{mdf-Shr-sys-2}, with coefficients depending on  $(\la,\SS)$ where $\SS=(h,V,A,B)$ and $h=g-I_d$. Under suitable gauge conditions, the auxiliary variables $(\la,\SS)$ are  shown to satisfy an elliptic system \eqref{ell-syst}, as well as a natural set of constraints. We conclude the section with a gauge formulation of our main result, see Theorem~\ref{GWP-MSS-thm}. For the detailed derivation, we refer to section 2 in \cite{HT}.

\subsection{The Riemannian metric \texorpdfstring{$g$}{} and the second fundamental form.} Let $(\Sigma^d,g)$ be a $d$-dimensional oriented manifold and let $(\R^{d+2},g_{\R^{d+2}})$ be $(d+2)$-dimensional Euclidean space. Let $\al,\be,\ga,\cdots \in\{1,2,\cdots,d\}$. Considering the immersion $F:\Sigma\rightarrow (\R^{d+2},g_{\R^{d+2}})$, we obtain the induced metric $g$ in $\Sigma$,
\begin{equation}       \label{g_metric}
	g_{\al\be}=\d_{x_{\al}} F\cdot \d_{x_{\be}} F.
\end{equation}
We denote the inverse of the matrix $g_{\al\be}$ by $g^{\al\be}$, i.e.
\begin{equation*}
g^{\al\be}:=(g_{\al\be})^{-1},\quad g_{\al\ga}g^{\ga\be}=\delta_{\al}^{\be}.
\end{equation*}

Let $\nab$ be the canonical Levi-Civita connection in $\Si$ associated with the induced metric $g$.
A direct computation shows that on the Riemannian manifold $(\Si,g)$ we have the Christoffel symbols
\begin{align*}
	\Gamma^{\ga}_{\al\be}=\ \frac{1}{2}g^{\ga\si}(\d_{\be}g_{\al\si}+\d_{\al}g_{\be\si}-\d_{\si}g_{\al\be})
	=\ g^{\ga\si}\d^2_{\al\be}F\cdot\d_\si F.
\end{align*}
Hence, the Laplace-Beltrami operator $\Delta_g$ can be written in the form
\begin{align*}
	\Delta_g f=&\ \tr\nab^2 f=g^{\al\be}(\d_{\al\be}^2f-\Gamma^{\ga}_{\al\be}\d_{\ga} f),
\end{align*}
for any twice differentiable function $f:\Sigma\rightarrow \R$. The curvature tensor $R$ on the Riemannian manifold $(\Sigma,g)$ is given by
\begin{equation*}             \label{R}
R_{\ga\al\be}^{\si}=\d_{\al} \Gamma_{\be\ga}^{\si}-\d_{\be} \Gamma_{\al\ga}^{\si} +\Gamma_{\be\ga}^m\Gamma_{\al m}^{\si}  -\Gamma_{\al\ga}^m\Gamma_{\be m}^{\si},\quad
R_{\al\be\ga\si}=g_{\mu\al}R_{\be\ga\si}^{\mu}.
\end{equation*}
We will also use the Ricci curvature
\begin{equation*}
\Ric_{\al\be}=R^{\si}_{\al\si\be}=g^{\si\ga}R_{\ga\al\si\be}.
\end{equation*}

Next, we compute the second fundamental form.
Let $\bar{\nab}$ be the Levi-Civita connection in $(\R^{d+2},g_{\R^{d+2}})$ and let $\bh$ be the second fundamental form for $\Sigma$ as an embedded manifold. Then by the Gauss relation we have
\begin{align*}
	\bh_{\al\be}=&\bh(\d_{\al},\d_{\be})=\bar{\nab}_{\d_{\al}} \d_{\be} F-F_{\ast}(\nab_{\d_{\al}}\d_{\be})
	=\d_{\al\be}^2 F-\Gamma_{\al\be}^{\ga} \d_{\ga} F.
\end{align*}
This gives the mean curvature $\bH$ at $F(x)$,
\begin{equation*}
	\bH=\tr_g \bh=\Delta_g F.
\end{equation*}
Hence, the $F$-equation in (\ref{Main-Sys}) is rewritten as
\begin{equation*}  
	(\d_t F)^{\perp}=J(F)\Delta_g F=J(F)g^{\al\be}(\d^2_{\al\be}F-\Gamma^{\ga}_{\al\be}\d_{\ga} F).
\end{equation*}
This equation is still independent of the choice of coordinates
in $\Sigma^d$.

\subsection{The complex structure equations}
This part is inspired by Gomez \cite{Gomez}.
We introduce a complex structure on the normal bundle $N\Sigma_t$. This is achieved
by choosing $\{\nu_1,\nu_2\}$ to be an orthonormal basis of $N\Si_t$ such that
\[
J\nu_1=\nu_2,\quad J\nu_2=-\nu_1.
\]
Note that such a choice is not unique.

The vectors $\{ F_1,\cdots,F_d,\nu_1,\nu_2\}$ form a frame at each point on the manifold $(\Sigma,g)$, where $F_{\al}=\d_{\al}F$ for $\al\in\{1,\cdots,d\}$. We define the tensors $\kappa_{\al\be}$, $\tau_{\al\be}$, the connection coefficients $A_{\al}$ and the temporal component $B$ of the connection in the normal bundle by
\[
\kappa_{\al\be}:=\d_{\al} F_{\be}\cdot\nu_1,\quad \tau_{\al\be}:=\d_{\al} F_{\be}\cdot\nu_2,\quad A_{\al}=\d_{\al}\nu_1\cdot\nu_2, \quad B=\d_t \nu_1\cdot \nu_2.
\]
We then define the complex vector field $m$ and  the complex second fundamental form tensor $\la_{\al\be}$ to be
\begin{equation*}
	m=\nu_1+i\nu_2,\quad \lambda_{\al\be}=\kappa_{\al\be}+i\tau_{\al\be},
\end{equation*}
and define the \emph{complex scalar mean curvature} $\psi$ to be the trace of $\lambda$,
\begin{equation}  \label{csmc}
    \psi=\tr \la= g^{\al\be}\la_{\al\be}.
\end{equation}
If we differentiate the frame, then we obtain a set of structure equations of the following type
\begin{equation}               \label{strsys-cpf}
\left\{\begin{aligned}
&\d_{\al}F_{\be}=\Gamma^{\ga}_{\al\be}F_{\ga}+\Re(\lambda_{\al\be}\bar{m}),\\
&\d_{\al}^A m=-\lambda^{\ga}_{\al} F_{\ga},
\end{aligned}\right.
\end{equation}
where $\d_{\al}^A=\d_{\al}+iA_{\al}$.

We then use the structure equations \eqref{strsys-cpf} to derive a set of constraints for $\la$ and $A$, and hence to obtain their elliptic equations.
Precisely, by \eqref{strsys-cpf} and the relations $\d_\al\d_\be F=\d_\be \d_\al F$, we obtain the Riemannian curvature and Ricci curvature
\begin{align}          \label{R-la}
R_{\si\ga\al\be}=\Re(\lambda_{\be\ga}\bar{\la}_{\al\si}-\la_{\al\ga}\bar{\la}_{\be\si}),\quad \Ric_{\ga\be}=\Re (\la_{\ga\be}\bar{\psi}-\la_{\ga\al}\bar{\la}_{\be}^{\al}),
\end{align}
as well as the Codazzi relations
\begin{equation}       \label{comm-la}
	\nab^A_{\al} \la^{\ga}_{\be}=\nab^A_{\be} \la^{\ga}_{\al}=\nab^{A,\ga}\la_{\al\be}.
\end{equation}
The structure equations \eqref{strsys-cpf} combined with the relations $\d_\al\d_\be m=\d_\be \d_\al m$ imply the compatibility condition for connection coefficients $A$
\begin{equation}          \label{cpt-AiAj-2}
\nab_{\al} A_{\be}-\nab_{\be} A_{\al}=\Im(\la^{\ga}_{\al}\bar{\la}_{\be\ga}).
\end{equation}

We state an elliptic system for the second fundamental form $\la$ in terms of $\psi$, using the Codazzi relations \eqref{comm-la} and $\psi=\tr \la$.
\begin{lemma}[Div-curl system for $\la$, Lemma 2.2 \cite{HT}] The second fundamental form $\la$ satisfies
    \begin{equation}     \label{div-curl-la}
    \nab^A_\al \la_{\be\ga}-\nab^A_\be \la_{\al\ga}=0,\quad
	      \nab^{A,\al}\la_{\al\be}=\nab^A_\be \psi.
\end{equation}
\end{lemma}
The second fundamental form $\la$ should also satisfy the constraint
\begin{equation}   \label{lambda-sim}
\la_{\al\be}=\la_{\be\al}.
\end{equation}

In order to both fix the gauge and obtain an elliptic system for $A$, we impose the Coulomb gauge condition
\begin{equation}\label{Coulomb}
    \nab^\al A_\al=0.
\end{equation}
We state the elliptic $A$-equations from the Ricci equations \eqref{cpt-AiAj-2}.
\begin{lemma}[Div-curl form for $A$] Under the Coulomb gauge condition \eqref{Coulomb}, the connection $A$ solves
    \begin{equation}      \label{Ellp-A}
    \nab_{\al} A_{\be}-\nab_{\be} A_{\al}=\Im(\la^{\ga}_{\al}\bar{\la}_{\be\ga}),\quad \nab^\al A_\al=0.
\end{equation}
\end{lemma}

As a corollary, we can derive a second order elliptic equation for $A$.
\begin{cor}
Under the Coulomb gauge condition \eqref{Coulomb} and harmonic coordinates \eqref{hm-coord}, the connection $A$ solves
    \begin{equation}      \label{Ellp-A1}
    \d_\ga \d^\ga A_\al=(\d_\ga g^{\ga\be}\d_\al -\d_\al g^{\ga\be}\d_\ga)A_\be-\d_\ga\Im(\la_{\al\si}\bar{\la}^{\ga\si}),
\end{equation}
\end{cor}
\begin{proof}
    By harmonic coordinates \eqref{hm-coord}, the div-curl system \eqref{Ellp-A} for $A$ can be rewritten as
    \begin{equation*}
        g^{\ga\be}\d_\al A_\be-\d^{\be} A_{\al}=\Im(\la^{\si}_{\al}\bar{\la}_{\si}^\ga),\quad \d^\al A_\al=0.
    \end{equation*}
    Then the equation \eqref{Ellp-A1} is obtained by applying $\d_\be$ to the first equation.
\end{proof}

\subsection{The elliptic equation for the metric \texorpdfstring{$g$}{} in harmonic coordinates} 
Here we take the next step towards fixing the gauge, by choosing to work in harmonic coordinates. Precisely, we will require the coordinate functions $\{x_{\al},\al=1,\cdots,d\}$ to be globally Lipschitz solutions of the elliptic equations
\begin{equation} \label{h-gauge}
\Delta_g x_{\al}=0.
\end{equation}
This determines the coordinates uniquely modulo time dependent affine transformations.
This remaining ambiguity will be removed later on by imposing suitable
boundary conditions at infinity.

Here, we will interpret the above harmonic coordinate condition at fixed time as an elliptic equation for the metric $g$.
The equations \eqref{h-gauge} can be expressed in terms of the Christoffel symbols $\Ga$, which  must satisfy the condition
\begin{equation}           \label{hm-coord}
g^{\al\be}\Ga^{\ga}_{\al\be}=0,\quad {\rm for}\ \ga=1,\cdots,d.
\end{equation}

In fact, we can obtain global harmonic coordinate by the smallness of $\d_x h$ in $H^s$ as follows. Here for a change of coordinate $y=x+\phi(x)$, we denote
\[  \tilde{F}(y)=F(x(y)),  \]
and denote its metric and Christoffel symbols as $\tilde{g}$ and $\tilde{\Ga}$, respectively.
\begin{lemma}[Existence of global harmonic coordinates, Proposition 8.1 \cite{HT}]
	Let $d\geq 3$, $s>\frac{d}{2}$, and
	$
	F:(\R^d_x,g)\rightarrow (\R^{d+2},g_{\R^{d+2}})
	$
	be an immersion with induced metric $g=I_d+h$. Assume that $\d_x h(x)$ is small in $H^{s}(dx)$, i.e.
	$\|\d_x h\|_{H^s}\leq \ep_0.$
	Then there exists a unique change of coordinates
	$y=x+\phi(x)$ with $\lim_{x\rightarrow\infty}\phi(x)=0$ and $\nab\phi$ uniformly small,
	such that the new coordinates $\{y_1,\cdots,y_d\}$ are global harmonic coordinates, namely,
	\begin{equation*}
	\tilde g^{\al\be}(y)\tilde{\Ga}_{\al\be}^\ga(y)=0,\quad \text{for any }y\in\R^d.
	\end{equation*}
	Moreover,
	\begin{equation*}  
	\|\d_x^{2}\phi(x)\|_{H^{s}(dx)}\lesssim \| \d_x h(x)\|_{H^{s}(dx)},
	\end{equation*}
	and, in the new coordinates $\{y_1,\cdots,y_d\}$,
	\begin{equation*}  
	\| \d_y\tilde h\|_{H^{s}(dy)}\lesssim \|  \d_x h\|_{H^{s}(dx)}.
	\end{equation*}
\end{lemma}

Under harmonic coordinate, the Ricci curvature formula \eqref{R-la} leads to an equation for the metric $g$:
\begin{lemma}[Elliptic equations of $g$, Lemma 2.4 \cite{HT}] In harmonic coordinates, the metric $g$ satisfies
	\begin{equation}\label{g-eq-original}
	\begin{aligned}
	g^{\al\be}\d^2_{\al\be}g_{\ga\si}=&\ [-\d_{\ga} g^{\al\be}\d_{\be} g_{\al\si}-\d_{\si} g^{\al\be}\d_{\be} g_{\al\ga}+\d_{\ga} g_{\al\be}\d_{\si} g^{\al\be}]\\
	&+2g^{\al\be}\Ga_{\si\al,\nu}\Ga^{\nu}_{\be\ga}-2\Re (\la_{\ga\si}\bar{\psi}-\la_{\al\ga}\bar{\la}_{\si}^{\al}).
	\end{aligned}
	\end{equation}
\end{lemma}

\subsection{The motion of the frame \texorpdfstring{$\{F_1,\cdots,F_d,m\}$}{} under (SMCF)} Here we derive the equations of motion for the frame, assuming that the immersion $F$ satisfies (\ref{Main-Sys}). Then we state the Schr\"odinger equation for mean curvature $\psi$ and the elliptic equations for advection fields $V$ and temporal connection coefficient $B$.

We begin by rewriting  the SMCF equations in the form
\begin{equation*}   
\d_t F=J(F)\bH(F)+V^{\ga} F_{\ga},
\end{equation*}
where $V^{\ga}$ is a vector field on the manifold $\Sigma$, which
in general depends on the choice of coordinates.
By the definition of $m$ and $\la_{\al\be}$, we get
\begin{equation}          \label{sys-cpf}
\d_t F=-\Im (\psi\bar{m})+V^{\ga} F_{\ga}.
\end{equation}

Applying $\d_{\al}$ to (\ref{sys-cpf}), by the structure equations (\ref{strsys-cpf}) and $m\bot F_{\al}=0$ we obtain
the equations of motion for the frame
\begin{equation}              \label{mo-frame}
\left\{\begin{aligned}
&\d_t F_{\al}=-\Im (\d^A_{\al} \psi \bar{m}-i\la_{\al\ga}V^{\ga} \bar{m})+[\Im(\psi\bar{\la}^{\ga}_{\al})+\nab_{\al} V^{\ga}]F_{\ga},\\
&\d^{B}_t m=-i(\d^{A,\al} \psi -i\la^{\al}_{\ga}V^{\ga} )F_{\al},
\end{aligned}\right.
\end{equation}
where $\d_t^B=\d_t+iB$.

From this we obtain the evolution equation for the metric $g$. Precisely, we denote
\begin{equation*}    
G_{\al\be}=\Im(\psi\bar{\la}_{\al\be})+\frac{1}{2}(\nab_\al V_\be +\nab_\be V_\al).
\end{equation*}
By the definition of the induced metric $g$ (\ref{g_metric}) and \eqref{mo-frame} we have
\begin{align}     \label{g_dt}
\d_t g_{\al\be}
=\ 2G_{\al\be},
\end{align}
and the evolutions of $g^{\al\be}$ and $\sqrt{\det g}$ are
\begin{equation*}
\d_t g^{\al\be}=-2G^{\al\be},\qquad
\d_t \sqrt{\det g}= \nab_\al V^\al\sqrt{\det g}.
\end{equation*}
These can give the evolution equations for Christoffel symbols
\begin{equation}   \label{dt-Ga}
     \d_t \Ga^\ga_{\al\be}=\nab_\al G^\ga_\be+\nab_\be G^\ga_\al-\nab^\ga G_{\al\be}.
     \end{equation}	
Moreover, by \eqref{g_dt} and \eqref{dt-Ga} we have
\begin{equation*}
    \d_t (g^{\al\be}\Ga_{\al\be}^\ga)
        = -2G^{\al\be}\Ga_{\al\be}^\ga+2\nab_\al\Im(\psi\bar{\la}^{\al\ga})+\De_g V^\ga+ \Re(\la^\ga_\si \bar{\psi}-\la_{\al\si}\bar{\la}^{\al\ga})V^\si.
\end{equation*}

So far, the choice of $V$ has been unspecified; it depends on the choice of coordinates
on our manifold as the time varies. However, once the latter is fixed via the harmonic coordinate condition (\ref{hm-coord}), we can also derive an elliptic equation for the advection field $V$:

\begin{lemma}[Elliptic equation for the vector field $V$, Lemma 2.5 \cite{HT}] Under the harmonic coordinate condition~\eqref{hm-coord}, the advection field $V$ solves
	\begin{equation}\label{Ellp-X}
	\begin{aligned}
	\De_g V^{\ga}=&\ -2\nab_\al\Im(\psi\bar{\la}^{\al\ga})-\Re (\la^{\ga}_{\si}\bar{\psi}-\la_{\al\si}\bar{\la}^{\al\ga})V^{\si}\\
	&\ +2(\Im(\psi\bar{\la}^{\al\be})+\nab^{\al}V^{\be})\Ga^{\ga}_{\al\be}.
	\end{aligned}
	\end{equation}
\end{lemma}

Next, from the equations \eqref{mo-frame} of motion for the frame we derive the main Schr\"{o}dinger equation and the second compatibility condition. The starting point is  the commutation relation
\begin{equation*} 
[\d^{B}_t,\d^A_{\al}]m=i(\d_t A_{\al}-\d_{\al} B)m,
\end{equation*}
which can be expanded, see \cite{HT}, equating the coefficients of the tangent vectors and of the normal vector $m$. Using the expressions 
(\ref{strsys-cpf}), (\ref{mo-frame}) for the derivatives of the frame,
this yields the evolution equation for $\la$
\begin{equation}\label{main-eq-abst}
\d^{B}_t\la^{\si}_{\al}+\la^{\ga}_{\al}(\Im(\psi\bar{\la}^{\si}_{\ga})+\nab_{\ga} V^{\si})=i\nab^A_{\al}(\d^{A,\si} \psi -i\la^{\si}_{\ga}V^{\ga} ),
\end{equation}
as well as the compatibility condition (curvature relation)
\begin{equation}\label{Cpt-A&B}
\d_t A_{\al}-\d_{\al} B = \Re(\la_{\al}^{\ga}\overline{\nab^A_{\ga}\psi})-\Im (\la^\ga_\al \bar{\la}_{\ga\si})V^\si.
\end{equation}

This in turn allows us to  use the Coulomb gauge condition \eqref{Coulomb} in order to obtain
an elliptic equation for $B$:

\begin{lemma}[Elliptic equation for $B$] The temporal connection coefficient $B$ solves
	\begin{equation}                \label{Ellip-B}
	\begin{aligned}
	\nab^{\ga}\nab_{\ga}B&=-\nab^\ga\nab_\si\Re(\la^\si_\ga \overline{\psi})+\frac{1}{2}\De_g|\psi|^2+\nab^{\ga}[\Im (\la^\si_\ga \bar{\la}_{\si\be})V^\be]\\
	&\quad +(2\Im(\psi\bar{\la}^{\be\ga})+\nab^{\be}V^{\ga}+\nab^{\ga}V^{\be})\d_{\be}A_{\ga}.
	\end{aligned}
	\end{equation}
\end{lemma}
\begin{proof}
    The equation \eqref{Ellip-B} is obtained by Lemma 2.6 in \cite{HT} and the following relation
     \[  \Re(\la_\ga^\si \overline{\nab^A_\si \psi})=\nab_\si\Re(\la_\ga^\si \overline{ \psi})-\Re(\nab^A_\ga\psi \overline{ \psi})=\nab_\si\Re(\la_\ga^\si \overline{ \psi})-\frac{1}{2}\nab_\ga|\psi|^2.  \]
\end{proof}

Finally, we use (\ref{main-eq-abst}) to derive the main equation, i.e. the Schr\"odinger equation for $\psi$. By (\ref{comm-la}),
contracting \eqref{main-eq-abst} yields
\begin{equation}       \label{main-eq-simp}
i(\d^{B}_t-V^{\ga}\nab^A_{\ga})\psi+\nab^A_{\al}\nab^{A,\al}\psi=-i\la^{\ga}_{\si}\Im(\psi\bar{\la}^{\si}_{\ga}).
\end{equation}

\subsection{The main result for modified Schr\"{o}dinger system from SMCF}

To conclude, under the Coulomb gauge condition  $\nab^{\al}A_{\al}=0$ and the harmonic coordinate condition $g^{\al\be}\Ga^{\ga}_{\al\be}=0$, by \eqref{main-eq-simp}, \eqref{div-curl-la}, \eqref{g-eq-original}, \eqref{Ellp-X}, \eqref{Ellp-A} and \eqref{Ellip-B}, we obtain the Schr\"{o}dinger equation for the complex mean curvature $\psi$
\begin{equation}        \label{mdf-Shr-sys-2}
\left\{
\begin{aligned}
    & i(\d^{B}_t-V^{\ga}\nab^A_{\ga})\psi+\nab^A_{\al}\nab^{A,\al}\psi=-i\la^{\ga}_{\si}\Im(\psi\bar{\la}^{\si}_{\ga}),
    \\
    & \psi(0) = \psi_0,
    \end{aligned}
\right.
\end{equation}
where the metric $g$, curvature tensor $\la$, the advection field $V$, connection coefficients $A$ and $B$ are determined at fixed time in an elliptic fashion via the following equations
\begin{equation}           \label{ell-syst}
	\left\{\begin{aligned}
	    &\nab^A_\al \la_{\be\ga}-\nab^A_\be \la_{\al\ga}=0,\quad
	      \nab^{A,\al}\la_{\al\be}=\nab^A_\be \psi,\\
		&\begin{aligned}
		g^{\al\be}\d^2_{\al\be}g_{\ga\si}=&\ [-\d_{\ga} g^{\al\be}\d_{\be} g_{\al\si}-\d_{\si} g^{\al\be}\d_{\be} g_{\al\ga}+\d_{\ga} g_{\al\be}\d_{\si} g^{\al\be}]\\
		&+2g^{\al\be}\Ga_{\si\al,\nu}\Ga^{\nu}_{\be\ga}-2\Re (\la_{\ga\si}\bar{\psi}-\la_{\al\ga}\bar{\la}_{\si}^{\al}),
		\end{aligned}\\
		&\begin{aligned}
	    \nab^{\al}\nab_{\al}V^{\ga}=&\ 2\nab_\al \Im(\la^{\al\ga}\bar{\psi})-\Re (\la^{\ga}_{\si}\bar{\psi}-\la_{\al\si}\bar{\la}^{\al\ga})V^{\si}\\
	    &+2(\Im(\psi\bar{\la}^{\al\be})+\nab^{\al}V^{\be})\Ga^{\ga}_{\al\be},
	    \end{aligned}\\
	    &\nab_{\al} A_{\be}-\nab_{\be} A_{\al}=\Im(\la^{\ga}_{\al}\bar{\la}_{\be\ga}),\quad \nab^\al A_\al=0,\\
		&\begin{aligned}\nab^{\ga}\nab_{\ga}B=&-\nab^\ga\nab_\si\Re(\la^\si_\ga \overline{\psi})+\frac{1}{2}\De_g|\psi|^2+\nab^{\ga}[\Im (\la^\si_\ga \bar{\la}_{\si\be})V^\be]\\
		&+(2\Im(\psi\bar{\la}^{\be\ga})+\nab^{\be}V^{\ga}+\nab^{\ga}V^{\be})\d_{\be}A_{\ga}.\end{aligned}
	\end{aligned}\right.
\end{equation}
Fixing the remaining degrees of freedom (i.e. the affine group for the choice of the
coordinates as well as the time dependence of the $SU(1)$ connection)
 we can assume that the  following conditions hold at infinity in an averaged sense:
 \begin{equation*}  
\la(\infty)=0,\quad g(\infty) = I_d, \quad V(\infty) = 0,\quad A(\infty) = 0, \quad B(\infty) = 0
 \end{equation*}
These are needed to insure  the unique solvability of the above elliptic equations
in a suitable class of functions. For the metric $g$ it will be useful to use the representation
\begin{equation*}
g = I_d + h
\end{equation*}
so that $h$ vanishes at infinity.

We note that the above elliptic system \eqref{mdf-Shr-sys-2} is accompanied by
a large family of compatibility conditions as follows:  \eqref{R-la}, \eqref{lambda-sim}, \eqref{cpt-AiAj-2}, \eqref{Coulomb}, \eqref{hm-coord}, \eqref{g_dt}
\eqref{main-eq-abst} and \eqref{Cpt-A&B}.
These conditions can all be shown to be satisfied for small solutions to the nonlinear system \eqref{mdf-Shr-sys-2}-\eqref{ell-syst}.

Now we recall the small data local well-posedness result for the (SMCF) system in \cite[Theorem 2.7]{HT} in terms of the above system:
\begin{thm}[Small data local well-posedness in the good gauge] \label{small-MSS-LWP}
	Let $s>\frac{d}{2}$, $d\geq 4$. Then there exists $\ep_0>0$ sufficiently small such that, for all initial data $\psi_0$ satisfying 
	\begin{equation*}
	    \|\psi_0\|_{H^{s}}\leq \ep_0, 
	\end{equation*} 
the modified Schr\"odinger system \eqref{mdf-Shr-sys-2}, with $(\la,h,V,A,B)$ determined via the
elliptic system \eqref{ell-syst}, is locally well-posed in $H^s$ on the time interval $I=[0,1]$. Moreover, the mean curvature satisfies
the bounds
	\begin{equation*}
	\lV \psi\rV_{l^2 \bX^s} +  \lV (\la,h,V,A,B)\rV_{\mathcal {E}^s}\lesssim \lV \psi_0\rV_{H^s}.
	\end{equation*}
    In addition, the mean curvature $\psi$ and the auxiliary functions $(\la,h,V,A,B)$ satisfy
	the constraints  \eqref{csmc}, \eqref{R-la}, \eqref{lambda-sim}, \eqref{cpt-AiAj-2}, \eqref{Coulomb} and \eqref{hm-coord} for any fixed time $t\in [0,1]$, and the evolutions \eqref{g_dt}, \eqref{main-eq-abst} and \eqref{Cpt-A&B}.
\end{thm}

Here the solution $\psi$ satisfies in particular the expected bounds
\[
\| \psi \|_{C[0,1;H^{s}]}\leq \|\psi\|_{l^2X^s} \lesssim \|\psi_0\|_{H^{s}}.
\]
The spaces $l^2 {\bX}^s$ and  ${\mathcal E}^s$, defined in \cite[Section 3]{HT}, contain a more complete description of the  full set of variables $\psi,\la,h,V,A,B$, which includes both Sobolev regularity and local energy bounds.
In the above theorem,  by well-posedness we mean a full Hadamard-type well-posedness, see \cite{IT}.

The main result of this paper is to extend the above local solution for small data to global for the (SMCF) system in Theorem~\ref{GWP-thm} in terms of the above system. The next theorem represents the harmonic/Coulomb gauge form of our main result in Theorem~\ref{GWP-thm}:

\begin{thm}[Small data global regularity in the good gauge]   \label{GWP-MSS-thm}
	Let $s_d$, $r_d$ be as \eqref{Reg-index} and \eqref{r_d} respectively for $d\geq 4$. Then there exists $\ep_0>0$ sufficiently small such that, for all initial data $\psi_0$ satisfying
	\begin{equation}\label{ini-data}
	    \|\psi_0\|_{H^{s_d}}\leq \ep_0, 
	\end{equation} 
the modified Schr\"odinger system \eqref{mdf-Shr-sys-2}, with $(\la,h,V,A,B)$ determined via the
elliptic system \eqref{ell-syst}, is globally well-posed in $H^{s_d}$. Moreover, the mean curvature satisfies the bound \eqref{psi-full-reg} and the scattering \eqref{Scatter}.
\end{thm}

This result is achieved by the following bootstrap proposition and continuity method.
\begin{prop}[Bootstrap proposition]   \label{bootstrap-prop}
		Let $s_d$, $r_d$ be as \eqref{Reg-index} and \eqref{r_d} respectively for $d\geq 4$.  Assume that $(\psi,\la,\SS)$ is a solution to \eqref{mdf-Shr-sys-2} and \eqref{ell-syst} on some time interval $[0,T]$, $T\geq 1$, with initial data satisfying the smallness assumption \eqref{ini-data}. Assume also that the solution satisfies the bootstrap hypothesis
		\begin{equation}\label{prop-Ass1}
		\|\psi\|_{S[0,T]}+\| \psi\|_{L_T^\infty H^{s_d}_x}\leq  C_0\|\psi_0\|_{H^{s_d}}.
		\end{equation}
	    Then the following improved bound holds:
		\begin{equation}\label{prop-result}
		\|\psi\|_{S[0,T]}+\| \psi\|_{L_T^\infty H^{s_d}_x}\leq  \frac{C_0}{2}\|\psi_0\|_{H^{s_d}},
		\end{equation}
		where $C_0>1$ is a large universal constant.
\end{prop}
In the remaining sections, we will focus on the proof of this proposition,
which splits in a modular fashion into an energy component and a Strichartz
component, as in Propositions~\ref{En-prop}, \ref{impbd-1} in the introduction.

\bigskip

\section{Function spaces and notations}   \label{Sec3}
We define the function spaces for the study of global solutions to the system \eqref{mdf-Shr-sys-2}-\eqref{ell-syst}.
First we introduce the usual Sobolev spaces $W^{s,p},\ H^s$ and the intrinsic Sobolev spaces $\mathsf H^k$ for tensors on $(\mathbb R^d; g)$. The gauge independent intrinsic norms  $\mathsf H^k$ are used in the energy estimates. Then we state a equivalence relation between the $H^k$ and $\mathsf H^k$ norms in the harmonic/Coulomb gauge.

For a function $u(t,x)$ or $u(x)$, let $\hat{u}=\FF u$ denote the Fourier transform in the spatial variable $x$. Fix a smooth radial function $\varphi:\R^d \rightarrow [0,1] $ supported in $\{x\in\R^d:|x|\leq 2\}$ and equal to 1 in $\{x\in\R^d:|x|\leq 1\}$, and for any $i\in \Z$, let
\begin{equation*}
	\varphi_i(x):=\varphi(x/2^i)-\varphi(x/2^{i-1}).
\end{equation*}
We then have the spatial Littlewood-Paley decomposition,
\begin{equation*}
	\sum_{i=-\infty}^{\infty}P_i (D)=1, \quad \sum_{i=0}^{\infty}S_i (D)=1,
\end{equation*}
where we use the differential operator $D:=\frac{1}{\sqrt{-1}}\d_x$ and $P_i$ localizes to frequency $2^i$ for $i\in \Z$, i.e,
\begin{equation*}
	 \FF(P_iu)=\varphi_i(\xi)\hat{u}(\xi).
\end{equation*}
and
\[S_0(D)=\sum_{i\leq 0}P_i(D),\quad S_i(D)=P_i(D),\ {\rm for}\ i>0.\]
For simplicity of notation, we set
\[
u_j=S_j u,\quad u_{\leq j}=\sum_{i=0}^j S_i u,\quad u_{\geq j}=\sum_{i=j}^{\infty} S_i u.
\]

We denote $W^{s,p}$ for $1\leq p\leq \infty$ as the usual Sobolev spaces, and denote $H^s:=W^{s,2}$.  
For simplicity of notation, on some time interval $[0,T]$, we define
\[
 \|u\|_{L^r_T W^{s,p}} = \|u\|_{L^r(0,T; W^{s,p})},\quad 1\leq r\leq \infty.
\]
For the solution $\psi$ of Schr\"odinger equation in \eqref{mdf-Shr-sys-2} and the related second fundamental form $\la$, we will be working primarily in $L^\infty_T H^{s_d}\cap S[0,T]$ for $s_d>\frac{d}{r_d}+2$. Here $S[0,T]$ are the Strichartz spaces defined by \eqref{S4} and \eqref{S5}. For convenience, corresponding to the $L^2_t$ component of the Strichartz norms, we also define the Sobolev norm at fixed time as
\begin{equation*}
    \|\psi\|_{\mathbf{str}}:=\|\psi\|_{W^{1,4}}+\|\psi\|_{W^{s_d-2,r_d}},\quad \text{for}\ d=4,
\end{equation*}
and 
\begin{equation*}
    \|\psi\|_{\mathbf{str}}:=\|\psi\|_{W^{s_d-2,r_d}},\quad \text{for}\ d\geq 5.
\end{equation*}
For the elliptic system \eqref{ell-syst}, at a fixed time we define the $\mathcal{H}^s$ norm as
\begin{equation*}
    \lV (h,V,A,B)\rV_{\mathcal H^s}=\lV|D|h\rV_{H^{s+1}} +\lV|D|V\rV_{H^{s}}+\lV A\rV_{H^{s+1}}+\lV|D|B\rV_{H^{s-1}}.
\end{equation*}

Next, we define the intrinsic Sobolev spaces $\mathsf H^k$ for integer $k\in \mathbb Z$. Since the Schr\"odinger equation \eqref{mdf-Shr-sys-2} is a quasilinear equations with variable coefficients $g$, we will use the space $\mathsf H^k$ to derive its energy estimates later. Let $g$ be a Riemannian metric on $\R^d$, and $A_\ga$ be a magnetic potential.
For any complex tensor $T=T^{\al_1\cdots\al_r}_{\be_1\cdots\be_s}dx^{\beta_1}\otimes...dx^{\beta_s}\otimes\frac{\partial }{\partial x^{\alpha_1}}\otimes...\otimes\frac{ \partial }{\partial x^{\alpha_r}}$, the covariant derivative is defined by
\[\nab^A_\ga T=\nab_\ga T+iA_\ga T,\]
where 
\begin{align}   \label{co_d}
\nab_\ga T^{\al_1\cdots\al_r}_{\be_1\cdots\be_s}=\d_\ga T^{\al_1\cdots\al_r}_{\be_1\cdots\be_s}+\sum_{i=1}^r \Ga^{\al_i}_{\ga\si} T^{\al_1\cdots\al_{i-1}\si\al_{i+1}\cdots\al_r}_{\be_1\cdots\be_s}-\sum_{j=1}^s \Ga^{\si}_{\ga\be_{j}} T^{\al_1\cdots\al_r}_{\be_1\cdots\be_{j-1}\si\be_{j+1}\cdots\be_s}.
\end{align}
We have
\begin{align*}
|\nab^A T|^2_g=g_{\al_1\al'_1}\cdots g_{\al_r\al'_r} g^{\be_1\be'_1}\cdots g^{\be_s\be'_s} \nab^A_\ga T^{\al_1\cdots\al_r}_{\be_1\cdots\be_s}\overline{\nab^{A,\ga} T^{\al'_1\cdots\al'_r}_{\be'_1\cdots\be'_s}}.
\end{align*}
Then the intrinsic Sobolev norm $\mathsf H^k$ for nonnegative integer $k\in \mathbb N$ is defined by
\begin{equation*}   
\|T\|_{\mathsf H^k}=\Big(\sum_{l=0}^k \int_\Sigma|\nab^{A,l} T|_g^2 d\mu\Big)^{1/2},
\end{equation*}
where volume form is $d\mu=\sqrt{\det g}dx$ and $\nab^{A,l}$ is the $l$-th order covariant derivative. By duality, we can also define the negative intrinsic Sobolev spaces as
\begin{equation*}
\|T\|_{\mathsf H^{-k}}=\sup_{\|U\|_{\mathsf H^k}\leq 1} \<T,U\>_{L^2}.
\end{equation*}

Under a suitable smallness assumption on the metric $h$ and the magnetic 
field $A$ we have the following equivalence relation between $H^k$ and $\mathsf H^k$ for a range of integers $k$.
\begin{lemma}
	Let $d\geq 3$ and $s>\frac{d}{2}$. Assume that $\|\d_x h\|_{H^{s-1}}+\|A\|_{H^{s-1}}\leq \ep$. Then for any integer $0\leq k\leq s$ we have the equivalent relation
	\begin{equation}       \label{equ-rlt}
	\| T\|_{\mathsf  H^k}\approx \|T\|_{H^k}.
	\end{equation}
\end{lemma}

\begin{proof}
	By covariant derivative \eqref{co_d}, schematically we write  $\nab^{A,k} T$ as
	\begin{equation}        \label{nabT=dT}
	\nab^{A,k} T=\partial^k T+\sum_{1\leq n\leq k}\ \sum_{l_1+\cdots+l_{n+1}= k-n}\d^{l_1}(\Ga+iA)\cdots \d^{l_n}(\Ga+iA)\cdot \d^{l_{n+1}}T.
	\end{equation}
	Then by the smallness of $\d_x h$ we have
	\begin{equation}\label{eq-bd1}
	\begin{aligned}
	&\|\nab^k T \|_{L^2(d\mu)}
	\lesssim  \||\nab^k T|_g\|_{L^2(dx)}\\
	\lesssim &\|\d^k T\|_{L^2}+\sum_{1\leq n\leq k}\ \sum_{l_1+\cdots+l_{n+1}= k-n}\|\d_x^{l_1}(\Ga+iA)\cdots \d_x^{l_n}(\Ga+iA)\cdot \d_x^{l_{n+1}}T\|_{L^2}.
	\end{aligned}
	\end{equation}
	In order to bound the second term above, it suffices to bound $\Ga^n T$ in $H^{k-n}$. For this we claim that
	\begin{align}        \label{claim}
	\| \Ga T \|_{H^{k'-1}}\lesssim \|\d_x h\|_{H^{s-1}}\|T\|_{H^{k'}}, \quad \text{ for any }1\leq k'\leq k.
	\end{align}
	Then by induction we have
	\begin{align*}
	\| \Ga^n T \|_{H^{k-n}}\lesssim \|\d_x h\|_{H^{s-1}} \| \Ga^{n-1} T \|_{H^{k-n+1}} \lesssim  \|\d_x h\|_{H^{s-1}}^n\|T\|_{H^{k}}.
	\end{align*}
	This combined with \eqref{eq-bd1} and the smallness of $\d_x h$ and $A$ in $H^{s-1}$ implies
	\begin{align*}
	\| T\|_{\mathsf H^k}\lesssim \|T\|_{H^k}+\ep \|T\|_{H^k}\lesssim \|T\|_{H^k}.
	\end{align*}
	
	We now return to prove the claim \eqref{claim}. Using a Littlewood-Paley decomposition and Bernstein's inequality we have
	\begin{align*}
	\| S_j(\Ga T) \|_{H^{k'-1}}\lesssim & \sum_{0\leq j_1\leq j}2^{(k'-1)j} \|\Ga_j\|_{L^2} 2^{dj_1/2} \|T_{j_1}\|_{L^2}+\sum_{0\leq j_1\leq j} 2^{dj_1/2} \|\Ga_{j_1}\|_{L^2} 2^{(k'-1)j}  \|T_{j}\|_{L^2}\\
	&+\sum_{j_1> j} 2^{(d/2+k'-1)(j-j_1)}  \|\Ga_{j_1}\|_{H^{d/2-1}}   \|T_{j_1}\|_{H^{k'}}\\
	:=&I_{1j}+I_{2j}+I_{3j}.
	\end{align*}
	For the first term we have
	\begin{align*}
	I_{1j}\lesssim &\textbf{1}_{<d/2}(k')\sum_{0\leq j_1\leq j}2^{(d/2-1+\de)j} \|\Ga_j\|_{L^2} 2^{(d/2-k'+\de)(j_1-j)} 2^{k' j_1}\|T_{j_1}\|_{L^2}\\
	&+\textbf{1}_{\geq d/2}(k')\sum_{0\leq j_1\leq j}2^{(k'-1+\de)j} \|\Ga_j\|_{L^2} 2^{(d/2-k'-\de)j_1} 2^{k' j_1}\|T_{j_1}\|_{L^2}\\
	\lesssim & \|\Ga_j\|_{H^{s-1}}\| T\|_{H^{k'}}.
	\end{align*}
	For the second term we have
	\begin{align*}
	I_{2j}\lesssim \sum_{0\leq j_1\leq j} 2^{j_1-j} 2^{(d/2-1)j_1} \|\Ga_{j_1}\|_{L^2} 2^{k'j}  \|T_{j}\|_{L^2}
	\lesssim \|\Ga\|_{H^{s-1}}\| T_j\|_{H^{k'}}.
	\end{align*}
	The last term $I_{3j}$ is bounded by
	\begin{align*}
	    I_{3j} \lesssim \sum_{j_1>j}2^{(d/2+k'-1)(j-j_1)}\|\Ga_{j_1}\|_{H^{d/2-1}}\|T\|_{H^{k'}}.
	\end{align*}
	Hence, these give
	\begin{align*}
	\|\Ga T\|_{H^{k'-1}}\lesssim & (\sum_{j\geq 0} \| S_j(\Ga T) \|_{H^{k'-1}}^2)^{1/2}\lesssim \|\Ga\|_{H^{s-1}}\| T\|_{H^{k'}}\lesssim \|\d_x h\|_{H^{s-1}}\| T\|_{H^{k'}}.
 	\end{align*}
 	Then the claim \eqref{claim} is obtained.

	Conversely, by \eqref{nabT=dT} we also have
	\begin{align*}
	\|\d^k T\|_{L^2}\lesssim &\|\nab^{A,k} T\|_{L^2}+\sum_{1\leq n\leq k}\ \sum_{l_1+\cdots+l_{n+1}= k-n}\|\d^{l_1}(\Ga+iA)\cdots \d^{l_n}(\Ga+iA)\cdot \d^{l_{n+1}}T\|_{L^2}\\
	\lesssim &\||\nab^{A,k} T|_g\|_{L^2(d\mu)}+\ep\| T\|_{H^k}.
	\end{align*}
	This implies
	\begin{align*}
	\| T\|_{H^k}\lesssim \|T\|_{\mathsf H^k}+\ep\| T\|_{H^k}\lesssim \|T\|_{\mathsf H^k},
	\end{align*}
which completes the proof of the lemma.
\end{proof}

\bigskip

\section{Elliptic estimates}\label{Ellip-sec}

In this section, we consider the elliptic system \eqref{ell-syst}.
Its solvability was already considered in \cite{HT} under the assumption
that $\psi$ is small in $H^s$. Here we prove some additional space-time bounds for $(\la,\SS)$, which are adapted to the Strichartz norm 
we will use later on. We begin by recalling the solvability  result in \cite{HT}:.

\begin{thm}[\cite{HT}, Theorem 4.1]\label{t:ell-fixed-time}
	Assume that $\psi$ is small in $H^s$ for $s > d/2$ and $d \geq 4$. Then the elliptic system \eqref{ell-syst} admits a unique small solution $(\la,\SS)$ in $\mathcal H^s$,
	with
	\begin{equation}   \label{ell-bd}
	\|(\la,\SS)\|_{H^s\times\mathcal H^s} \lesssim \| \psi\|_{H^s}.
	\end{equation}
	Moreover, for the linearization of the solution map above we also have the bound:
	\begin{align}\label{ell-lin}
	    \|(\la_{lin},\SS_{lin})\|_{H^{\si}\times\mathcal H^{\si}} \lesssim \| \psi_{lin}\|_{H^{\si}},\quad \si \in (\frac{d}{2}-3,s].
	\end{align}
\end{thm}

Here we will supplement the above result with an an additional set of  estimates:
\begin{lemma}\label{Lt2-lem}
Let
\begin{equation*}
\frac{d}{r_d}< \si\leq s_d-2, \qquad \si_1\leq s_d+2-\frac{d(d-2)}{2(d-1)}. 
\end{equation*}
Let $\psi$ be defined in the interval $[0,T]$ and satisfying the hypothesis of Proposition \ref{bootstrap-prop}.
Then  we have
\begin{equation}    \label{L2t-SS}
\|\la\|_{S[0,T]}+\|h\|_{L^2_T W^{\si_1,2(d-1)}}+\| (0,V,A,B)\|_{L^2_T\HH^{s_d}}\lesssim \|\psi\|_{S[0,T]}.
\end{equation}
and 
\begin{equation}    \label{L1t-V}
	\|\d_x V\|_{L^1_TW^{\si,r_d}}+\|\d_x^2 V\|_{L^1_T L^d}\lesssim \|\psi\|_{S[0,T]}^2.
\end{equation}
In addition, in dimension $d=4$ we have
\begin{align}          \label{L2t-hW14}
    \|h\|_{L^2_TW^{1,4}}\lesssim \|\psi\|_{L^2_TW^{1,4}}\|\psi\|_{L^\infty_T H^{1}}.
\end{align}
\end{lemma}

We remark that in essence this is a fixed time result, where 
the evolution equation for $\psi$ is never used. What we prove in effect is the corresponding bound at fixed time where all the $L^2_T$ norms are dropped.

\begin{proof}
	\emph{Step 1: The estimate for $\la$ in \eqref{L2t-SS}.}
	Here we use the div-curl system \eqref{div-curl-la}, which we write
	schematically in the form
	\begin{align*}
	&\d_\al \la_{\al\be}=\d_\be \psi+A\psi+h\d_x\la+\d_x h\la,\\
	&\d_\al \la_{\be\ga}-\d_{\be}\la_{\al\ga}=A\la+\d_x h \la.
	\end{align*}
	By the relation
	\begin{equation*}  
	 \widehat{\la}(\xi)=|\xi|^{-2}(\widehat{\la}\cdot \xi)\xi+|\xi|^{-2}(\widehat{\la}\xi^{\top}-\xi \widehat{\la}^{\top})\cdot\xi,
	\end{equation*}
	we have
	\begin{align*}
	\| \la\|_{W^{\si,r_d}} 
	&\lesssim  \| \RR(\RR\cdot \la)\|_{W^{\si,r_d}}+\| \RR(\RR_\al \la_{\be\ga}-\RR_\be\la_{\al\ga})\|_{W^{\si,r_d}}\\
	&\lesssim  \| \psi\|_{W^{\si,r_d}}+\| |D|^{-1}(A\psi+A\la+h\d_x\la+\d_x h\la)\|_{W^{\si,r_d}}.
	\end{align*}
	By Sobolev embeddings and \eqref{ell-bd} we can estimate
	\begin{align*}
	\||D|^{-1}(h\d_x \la+\d_x h\la)\|_{W^{\si,r_d}}
	&\lesssim \|\RR(h \la)\|_{W^{\si,r_d}}+\||D|^{-1}(\d_x h \la)\|_{W^{\si,r_d}}\\
	&\lesssim \|h\|_{W^{\si,r_d}}\|\la\|_{L^\infty}+\|h\|_{L^\infty}\|\la\|_{W^{\si,r_d}}\\
	&\quad +\||D|^{-1}P_{\leq 0}(\d_x h\la)\|_{L^{r_d}}+\|\d_x h \la\|_{W^{\si-1,r_d}}\\
	&\lesssim  \| \d_x h\|_{H^s}\|\la\|_{W^{\si,r_d}}+\|\d_x h\la\|_{L^2}\\
	&\lesssim  \ep_0\|\la\|_{W^{\si,r_d}}.
	\end{align*}
	Similarly, since $\psi=\Tr \la$ we can bound the other terms by
	\begin{align*}
	\| |D|^{-1}(A\la)\|_{W^{\si,r_d}}&
	\lesssim  \|A\la\|_{L^2}+\| A\la\|_{W^{\si-1,r_d}}\\&
	\lesssim  \| \d_x A\|_{H^s}\|\la\|_{W^{\si,r_d}}+\|A\|_{W^{\si-1,r_d}} \|\la\|_{L^\infty}\\&
	\lesssim \| \d_x A\|_{H^s}\|\la\|_{W^{\si,r_d}}\\&
	\lesssim \ep_0\|\la\|_{W^{\si,r_d}}.
	\end{align*}
	Hence, from these estimates we obtain
	\begin{align}       \label{L2t-la}
	\|\la\|_{L^2_TW^{\si,r_d}}\lesssim \|\psi\|_{L^2_TW^{\si,r_d}}.
	\end{align}
Similarly, in dimensions 4 we also have
\begin{equation}     \label{L2t-laW14}
	\|\la\|_{L^2_TW^{1,4}}\lesssim \|\psi\|_{L^2_TW^{1,4}}.
\end{equation}

\medskip

\emph{Step 2: The estimate for the metric $g$ in \eqref{L2t-SS} and \eqref{L2t-hW14}.}
It suffices to consider the following schematic form
of the equations \eqref{g-eq-original}:
\begin{equation*}
	\De h=h\d_x^2 h+\d_x h\d_x h+h\d_x h\d_x h+\la^2.
	\end{equation*}
	For the first three terms, we use Sobolev embeddings and H\"older's inequality to estimate
	\begin{align*}
	&\|\De^{-1}( h\d_x^2 h+\d_x h\d_x h+h\d_x h\d_x h)\|_{W^{\si_1,2(d-1)}}\\&
	\lesssim \| h\d_x^2 h+\d_x h\d_x h+h\d_x h\d_x h\|_{L^{\frac{2d(d-1)}{5d-4}}}+\| h\d_x^2 h+\d_x h\d_x h+h\d_x h\d_x h\|_{W^{\si_1-2,2(d-1)}}\\&
	\lesssim (1+\|\d_x h\|_{H^{s+1}})\|\d_x h\|_{H^{s+1}}\| h\|_{W^{\si_1,2(d-1)}}\lesssim \ep_0\| h\|_{W^{\si_1,2(d-1)}}.
	\end{align*}
	For the last term, by Sobolev embeddings we have
	\begin{align*}
	\| \De^{-1}(\la^2)\|_{W^{\si_1,2(d-1)}}&\lesssim \| \la^2\|_{L^{\frac{2d(d-1)}{5d-4}}}+ \|\la^2\|_{W^{\si_1-2,2(d-1)}}\\&
	\lesssim \|\la\|_{L^{d/2}}\|\la\|_{L^{2(d-1)}}+\|\la\|_{W^{\si_1-2,2(d-1)}}\|\la\|_{L^\infty}\\&
	\lesssim \| \la\|_{H^s}\| \la\|_{W^{\si,r_d}}
	\lesssim  \ep_0 \| \la\|_{W^{\si,r_d}}.
	\end{align*}
	Hence, by the above estimates and \eqref{L2t-la} we obtain
	\begin{equation*}
	\|h\|_{L^2_T W^{\si_1,2(d-1)}}\lesssim \ep_0 \| \la\|_{L^2_TW^{\si,r_d}}\lesssim \ep_0 \| \psi\|_{L^2_TW^{\si,r_d}}.
	\end{equation*}
	In the same way, from \eqref{L2t-laW14} we also obtain the bound \eqref{L2t-hW14} in dimension $d=4$.

	\medskip
	
    \emph{Step 3: The estimate for the advection field $V$ and the connection coefficients $A$ in \eqref{L2t-SS}.}
    Again it suffices to consider the following schematic form
    of the equations
    \eqref{Ellp-X}, \eqref{Ellp-A1}
    \begin{align*} 
    \De V&= h\d_x^2 V+\d_x h\d_x V+\d_x h\d_x h V+\la^2 (A+V+\d_x h)+\d_x(\la^2),\\
    \De A&=\d_x(\la^2)+\d_x(h\d_x A).
    \end{align*}
    The estimates for $V$ and $A$ are similar, so we only prove the bound for $V$.

    As in the proof of \eqref{ell-bd}, we bound the first three terms on the right by
    \begin{align*}
    \| |D|^{-1}(h\d_x^2 V+\d_x h\d_x V+\d_x h\d_x h V)\|_{H^{s_d}}\lesssim &\ (1+\|\d_x h\|_{ H^{s_d+1}})\|\d_x h\|_{ H^{s_d+1}}\|\d_x V\|_{ H^{s_d}}\\
    \lesssim &\  \ep_0 \|\d_x V\|_{ H^{s_d}}.
    \end{align*}
	For the forth term, by Sobolev embeddings we have
	\begin{align*}
	\| |D|^{-1}(\la^2(A+V+\d_x h))\|_{H^{s_d}}&\lesssim \| \la^2(A+V+\d_x h)\|_{L^{\frac{2d}{d+2}}}+\| \la^2(A+V+\d_x h)\|_{H^{s_d}}\\&
	\lesssim \|\la^2\|_{H^{s_d}}\|(\d_x A,\d_x V,\d_x h)\|_{H^{s_d}}\\&
	\lesssim  \ep_0^2 \|\la\|_{W^{\si,r_d}}.
	\end{align*}
	For the last term, we also have
	\begin{align*}
	\|\mathcal R(\la^2)\|_{H^{s_d}}\lesssim  \| \la\|_{H^{s_d}}\|\la\|_{L^\infty}
	\lesssim \ep_0 \|\la\|_{W^{\si,r_d}}.
	\end{align*}
	Hence, we obtain
	\begin{align*}
	\|\d_x V\|_{H^{s_d}}\lesssim \ep_0 \|\la\|_{W^{\si,r_d}}.
	\end{align*}

	\medskip

	\emph{Step 4: The estimate for $B$ in \eqref{L2t-SS}.}
	Again it suffices to consider the schematic form of the equation \eqref{Ellip-B}:
	\begin{align*}
	\De B&=h\d_x^2 B+\d_x(\la\d_x \la)+\la^2(\d_x A+\d_x V+\d_x h(V+A))+\la\d_x\la(V+A+\d_x h)\\
	&\quad +\d_x V\d_x A+\d_x h V\d_x A.
	\end{align*}
	By Sobolev embeddings and \eqref{ell-bd}, we obtain
	\begin{align*}
	\|\d_x B\|_{L^2H^{s_d-1}}\lesssim \ep_0(\|A\|_{L^2H^{s_d+1}}+\|\d_x V\|_{L^2H^{s_d}})+\ep_0\|\la\|_{L^2_tW^{\si,r_d}}\lesssim \ep_0\|\psi\|_{L^2_tW^{\si,r_d}}.
	\end{align*}
	The proof of this bound is similar to the above \emph{steps}, and we omit the detail.
	
	\medskip 
	
	\emph{Step 5: The estimates for $V$ in \eqref{L1t-V}.}
	It suffices to consider the form
	\begin{align} \label{V-general}
		\De V&= h\d_x^2 V+\d_x h\d_x V+\d_x h\d_x h V+\la^2 (A+V+\d_x h)+\d_x(\la^2).
	\end{align}

	First, we prove that 
	\begin{equation*}
		\|\d_x V\|_{L^1_TW^{\si_d,r_d}}\lesssim \|\la\|_{S[0,T]}^2+\ep_0 \|\d_x V\|_{L^1_TW^{\si_d,r_d}}. 
	\end{equation*}
    This implies the bound \eqref{L1t-V} for the term $\|\d_x V\|_{L^1_TW^{\si_d,r_d}}$.
    
	By $V$-equation and Sobolev embeddings we have
	\begin{align*}
		\|\d_x V\|_{L^1_T W^{k_0-2,r_d}}&\lesssim \||D|^{-1}\big[ h\d_x^2 V+\d_x h\d_x V+\d_x h\d_x h V \big]\|_{L^1_T W^{k_0-2,r_d}}\\&
		\quad +\||D|^{-1}\big[ \la^2 (A+V+\d_x h)  \big]\|_{L^1_T W^{k_0-2,r_d}}+\|\mathcal R(\la^2)  \|_{L^1_T W^{k_0-2,r_d}}\\&
		:= I+II+III,
	\end{align*}
	where $\mathcal R$ is Risez transform.
	For the first term, we easily have
	\begin{align*}
		I&\lesssim \| P_{\leq 0}( h\d_x^2 V+\d_x h\d_x V+\d_x h\d_x h V)\|_{L_T^1 L^2}\\&
		\quad +\| h\d_x^2 V+\d_x h\d_x V+\d_x h\d_x h V\|_{L_T^1 W^{k_0-3,r_d}}\\&
		\lesssim (1+\|\d_x h\|_{L_T^\infty H^{k_d}})(\|\d_x h\|_{L^\infty H^{k_d}}\|\d_x V\|_{L^1_T W^{k_d-2,r_d}}+\|\d_x h\|_{L^2 W^{k_d-2,r_d}}\|\d_x V\|_{L^2_T H^{k_d}})\\&
		\lesssim \ep_1 \|\d_x V\|_{L^1_T W^{k_d-2,r_d}}+\|\la\|_{S[0,T]}^2.
	\end{align*}
	For the second term, by Sobolev embeddings we have
	\begin{align*}
		II&	\lesssim \|P_{\leq 0}|D|^{-1}[\la^2(A+V+\d_x h)]\|_{L^{r_d}}+\|P_{> 0}[\la^2(A+V+\d_x h)]\|_{L^1 W^{\si_d-1,r_d}}\\
		&\lesssim \|\la\|_{L^2L^{r_d}}\|\la\|_{L^2L^\infty}\|A+V+\d_xh\|_{L^d}+\|\la\|_{L^1L^\infty}^2\|\nab (A+V+\d_x h)\|_{L^\infty H^{s_d-1}}\\
		&\quad +\|\la\|_{L^2W^{\si_d,r_d}}\|\la\|_{L^2L^\infty} \|A+V+\d_x h\|_{L^\infty L^\infty}\\
		&\lesssim \|\la\|_{L^2W^{\si_d,r_d}}^2\|(\d_x A,\d_x V,\d_x h)\|_{L^\infty H^{s_d}}\\
		&\lesssim \ep_0\|\la\|_{L^2W^{\si_d,r_d}}^2.
	\end{align*}
	For the last term, we also have
	\begin{align*}
		III\lesssim \|\la\|_{L_T^2L^\infty}\|\la\|_{L_T^2 W^{k_0-2,r_d} }\lesssim \|\la\|_{L_T^2 W^{k_0-2,r_d} }^2\lesssim \ep_1^2.
	\end{align*}
	Hence, we give the bound \eqref{L1t-V} for $\d_x V$.
	
	\medskip 
	
	Next, we prove that 
	\begin{align}   \label{d2V-key}
		\|\d^2_x V\|_{L^1L^d}\lesssim \ep_0 \|\d_x^2 V\|_{L^d}+ \|\la\|_{S[0,T]}^2
	\end{align}
	From the general form \eqref{V-general}, we use \eqref{L2t-SS}, \eqref{L1t-V} for $\d_x V$ to bound the first three terms 
	\begin{align*}
		&\|h\d_x^2 V+\d_x h\d_x V+\d_x h\d_x h V\|_{L^1L^d}\\
		&\lesssim \|\d_x h\|_{L^\infty H^{s_d}}\|\d_x^2 V\|_{L^1L^d}+\|\d_x h\|_{L^\infty L^d}\|\d_x V\|_{L^1W^{\si_d,r_d}}+\|\d_x h\|_{L^2L^\infty}^2 \|\d_x V\|_{L^\infty H^{s_d-2}}\\
		&\lesssim \ep_0 \|\d_x^2 V\|_{L^d}+\ep_0 \|\la\|_{S[0,T]}^2.
	\end{align*}
	We bound the last two terms by
	\begin{align*}
		&\|\la^2(A+V+\d_x h)\|_{L^1L^d}+\|\d_x(\la^2)\|_{L^1L^d}\\
		&\lesssim \|\la\|_{L^2L^\infty}^2\|(\d_x A,\d_x V,\d_x h)\|_{L^\infty H^{s_d-1}}+\|\la\|_{S[0,T]}^2\\
		&\lesssim \|\la\|_{S[0,T]}^2(1+\ep_0).
	\end{align*}
    Then the desired bound \eqref{d2V-key} follows, and we obtain the bound \eqref{L1t-V} for $\d^2_x V$.
\end{proof}

Finally, we turn our attention to the linearization of the elliptic system \eqref{ell-syst}. This has already been studied in \cite{HT} in nonnegative Sobolev spaces. However, for our global estimates here we need instead 
to work with the linearized equation in $H^{-1}$. For this case, the elliptic estimates are as follows:

\begin{prop}
    With the notation and hypothesis in \propref{bootstrap-prop}, for the linearized equations of \eqref{ell-syst} we have
    \begin{align}  \label{la-lin}
        \|\la_{lin}\|_{L^\infty H^{-1}}&\lesssim \|\psi_{lin}\|_{L^\infty H^{-1}},\\   \label{linEst}
        \|\d_x h_{lin}\|_{L^2L^2}+\|A_{lin}\|_{L^2L^{2}}+\|V_{lin}\|_{L^2L^{2}}+\|B_{lin}\|_{L^2H^{-1}}&\lesssim \|\psi_{lin}\|_{L^\infty H^{-1}}\|\psi\|_{S[0,T]}.
    \end{align}
\end{prop}

\begin{proof}
    \emph{Step 1: Prove the $h_{lin}$ bound }
    \begin{equation}    \label{h-lin}
         \|\d_x h_{lin}\|_{L^2L^2}\lesssim \|\la_{lin}\|_{L^\infty H^{-1}}\|\la\|_{L^2W^{\si_d,r_d}}.
    \end{equation}
    
    For the $h$-equations in \eqref{ell-syst}, we consider the general form
    \begin{align*}
        \De h=h\d_x^2 h+\d_x h\d_x h+h\d_x h\d_x h+\la^2.
    \end{align*}
    For the term $\la_{lin}\la$, by Sobolev embeddings we have
    \begin{equation}    \label{h-la2}
    \begin{aligned}
        &\||D|^{-1}(\la_{lin}\la)\|_{L^2}\\
        &\lesssim \||D|^{-1}(\la_{lin,\leq 0}\la)\|_{L^2}+\||D|^{-1}(\la_{lin,>0}\la)\|_{L^2}\\
        &\lesssim \|\la_{lin,\leq 0}\|_{L^2}\|\la\|_{L^d}+\||D|^{-1}(|D|^{-1}\la_{lin,>0}|D|\la)+|D|^{-1}\la_{lin,>0}\la\|_{L^2}\\
        &\lesssim \|\la_{lin,\leq 0}\|_{L^2}\|\la\|_{L^d}+\||D|^{-1}\la_{lin,>0}\|_{L^2}(\||D|\la\|_{L^d}+\|\la\|_{L^\infty})\\
        &\lesssim \|\la_{lin}\|_{H^{-1}}\|\la\|_{\mathbf{str}}.
    \end{aligned}
    \end{equation}
    For the term $h_{lin} \d^2 h$, we also have
    \begin{align*}
        \||D|^{-1}(h_{lin} \d^2 h)\|_{L^2}&\lesssim \|h_{lin} \d h\|_{L^2}+\||D|^{-1}(\d h_{lin} \d h)\|_{L^2}\\
        &\lesssim \|h_{lin}\|_{L^{\frac{2d}{d-2}}}\|\d h\|_{L^d}+\||D| h_{lin}\|_{L^2}\|\d h\|_{L^d}\\
        &\lesssim \|\d h_{lin}\|_{L^2}\|\d h\|_{L^d}.
    \end{align*}
    The other terms are controlled at the same way. Hence, by \eqref{prop-Ass1} we obtain
    \begin{align*}
        \|\d h_{lin}\|_{L^2}\lesssim \ep_0\|\d h_{lin}\|_{L^2}+\|\la_{lin}\|_{H^{-1}}\|\la\|_{\mathbf{str}}.
    \end{align*}
    This implies the bound \eqref{h-lin}.
    
    \medskip 
    
    \emph{Step 2: Prove the bound }
    \begin{equation}    \label{A,V-lin}
        \|A_{lin}\|_{L^2L^{2}}+\|V_{lin}\|_{L^2L^{2}}\lesssim \|\la_{lin}\|_{L^\infty H^{-1}}\|\la\|_{S[0,T]}.
    \end{equation}
    
    The estimates of $V_{lin}$ and $A_{lin}$ are similar, so we only prove the first one. For the $V$-equation, we consider the form
    \begin{align*}
    \De V= h\d_x^2 V+\d_x h\d_x V+\d_x h\d_x h V+\la^2 (A+V+\d_x h)+\d_x(\la^2).
    \end{align*}
    For $\d(\la^2)$, we have the bound \eqref{h-la2}. For the term $\la^2 (A+V+\d_x h)$, we have 
    \begin{align*}
        \|\De^{-1}(\la^2 V_{lin})\|_{L^2}\lesssim \|\la^2 V_{lin}\|_{L^{\frac{2d}{d+4}}}\lesssim \|\la\|_{L^d}^2\|V_{lin}\|_{L^2},
    \end{align*}
    and 
    \begin{align*}
        &\|\De^{-1}(\la_{lin}\la V)\|_{L^2}\\
        &\lesssim \|\De^{-1}(\<D\>^{-1}\la_{lin} \<D\>\la V+\<D\>^{-1}\la_{lin} \la \<D\>V)+\De^{-1}\<D\>(\<D\>^{-1}\la_{lin} \la V\|_{L^2}\\
        &\lesssim \|\<D\>^{-1}\la_{lin} \<D\>\la V+\<D\>^{-1}\la_{lin} \la \<D\>V\|_{L^{\frac{2d}{d+4}}}+\|\<D\>^{-1}\la_{lin} \|_{L^2} \|\la\|_{L^d\cap L^\infty}\| V\|_{L^d}\\
        &\lesssim \|\la_{lin} \|_{H^{-1}} \|\la\|_{\mathbf{str}}\| V\|_{L^d}.
    \end{align*}
    For the term $h\d^2 V$, we have
    \begin{align*}
        \|\De^{-1}(h_{lin} \d^2 V)\|_{L^2}
        &= \|\De^{-1}(\d h_{lin} \d V)+|D|^{-1}(h_{lin}\d V)\|_{L^2}\\
        &\lesssim \|\d h_{lin}\|_{L^2}\|\d V\|_{H^s},
    \end{align*}
    and 
    \begin{align*}
        \|\De^{-1}(h \d^2 V_{lin})\|_{L^2}
        &= \|\De^{-1}(\d^2 h V_{lin} )+|D|^{-1}(\d h V_{lin} )+h V_{lin}\|_{L^2}\\
        &\lesssim \|\d h\|_{H^s}\|\d V_{lin} \|_{L^2}.
    \end{align*}
    The second and third term are bounded similarly. Hence, we obtain
    \begin{align*}
        \|V_{lin}\|_{L^2}\lesssim \ep_1(\|V_{lin}\|_{L^2}+\|A_{lin}\|_{L^2})+\|\la_{lin}\|_{H^{-1}}\|\la\|_{\mathbf{str}}.
    \end{align*}
    In the same way, we also have
    \begin{align*}
        \|A_{lin}\|_{L^2}\lesssim \|\la_{lin}\|_{H^{-1}}\|\la\|_{\mathbf{str}}.
    \end{align*}
    These two estimates imply the desired bound \eqref{A,V-lin}.

     \medskip 
     
     \emph{Step 3: Prove the $\la_{lin}$ bound \eqref{la-lin}.}
     
     As before, it suffices to consider the simplified form
     of the div-curl system for $\lambda$, namely
     \begin{align*}
	 &\d_\al \la_{\al\be}=\d_\be \psi+A\psi+h\d_x\la+\d_x h\la,\\
	 &\d_\al \la_{\be\ga}-\d_{\be}\la_{\al\ga}=A\la+\d_x h \la.
	\end{align*}
	For the term $A\la$ we have
	\begin{align*}
	    \||D|^{-1}(A_{lin}\la)\|_{H^{-1}}\lesssim \|A_{lin}\|_{L^2}\|\la\|_{L^d},
	\end{align*}
	and 
	\begin{align*}
	    \||D|^{-1}(A\la_{lin})\|_{H^{-1}}&\lesssim \|A\|_{L^d}\|\la_{lin,\leq 0}\|_{L^2}+\||D|^{-1}(|D|A|D|^{-1}\la_{lin,>0})+A|D|^{-1}\la_{lin,>0}\|_{L^2}\\
	    &\lesssim \|\d A\|_{H^s}\|\la_{lin}\|_{H^{-1}}.
	\end{align*}
	The other terms are controlled by
	\begin{align*}
	    \|\psi_{lin}\|_{H^{-1}}+\|\d_x h_{lin}\|_{L^2}\|\la\|_{H^s}+\|\d_x h\|_{H^{s}}\|\la_{lin}\|_{H^{-1}}.
	\end{align*}
	Then these estimates combined with \eqref{h-lin} and \eqref{A,V-lin} yield
	\begin{align*}
	    \|\la_{lin}\|_{H^{-1}}\lesssim \|\psi_{lin}\|_{H^{-1}}+\ep_0 \|\la_{lin}\|_{H^{-1}}.
	\end{align*}
	This implies the bound \eqref{la-lin}.
	
	\medskip 
	
	\emph{Step 4: Prove the bound }
     \begin{equation}     \label{B-lin}
         \|B_{lin}\|_{L^2H^{-1}}\lesssim \|\la_{lin}\|_{L^\infty H^{-1}}\|\la\|_{S[0,T]}.
     \end{equation}
     
     For the $B$-equation we consider the general form
     \begin{align*}
	 \De B& =h\d_x^2 B+\d^2_x(\la^2)+\d_x(\la^2)(\d_x h+V)+\la^2(\d^2_x h+\d_x h\d_x h+\d_x A+\d_x V+\d_x h V)\\
	 &\quad +\d_x V\d_x A+\d_x h V\d_x A.
	\end{align*}
     For the second term $\d^2(\la^2)$ we have
     \begin{align*}
         \|\la_{lin}\la\|_{H^{-1}}\lesssim \|\la_{lin}\|_{H^{-1}}\|\la\|_{\mathbf{str}}.
     \end{align*}
     For the third term, we have
     \begin{align*}
         \|\De^{-1}\big(\d_x(\la_{lin}\la) V\big)\|_{H^{-1}}&\lesssim \|\la_{lin}\|_{H^{-1}}\|\la\|_{\mathbf{str}}\|\d_x V\|_{H^s}.
     \end{align*}
     and 
     \begin{align*}
         \|\De^{-1}\big(\d_x(\la^2) V_{lin}\big)\|_{H^{-1}}&\lesssim \|\la\|^2_{\mathbf{str}}\| V_{lin}\|_{L^2}.
     \end{align*}
     We control the other terms at the same way, then by \eqref{h-lin}, \eqref{A,V-lin} and \eqref{prop-Ass1} we obtain 
     \begin{align*}
         \|B_{lin}\|_{H^{-1}}\lesssim \ep_0\|B_{lin}\|_{H^{-1}}+\|\la_{lin}\|_{H^{-1}}\|\la\|_{\mathbf{str}}.
     \end{align*}
     This gives the bound \eqref{B-lin}.
     
     In conclusion, from \eqref{h-lin}, \eqref{A,V-lin}, \eqref{B-lin} and \eqref{la-lin} we obtain the second bound \eqref{linEst}.
\end{proof}

\bigskip


\section{Energy estimates} \label{Sec-Energy}
Here we consider the Schr\"odinger equation \eqref{mdf-Shr-sys-2}, and prove the energy estimates in \propref{En-prop} as well as an energy estimate of linearized Schr\"odinger equation. These will be needed in order to prove energy bounds \eqref{prop-result} in fractional Sobolev spaces.
For two tensors ${\mathbb T}$ and $\widetilde{{\mathbb T}}$, we denote ${\mathbb T}*\widetilde{\mathbb T}$ the bilinear combination of them.

To start with, we define the energy functional as follows. Let the metric $g$ and connection $A$ be (part of) the solutions to the elliptic equations \eqref{ell-syst}. For any nonnegative integer $k\in \N$, we define $E^k(\psi)$ as 
\begin{equation}     \label{def-energy}
    E^k(\psi):=\|\psi\|_{\mathsf H^k}^2=\Big(\sum_{l=0}^k \int_\Sigma|\nab^{A,l} \psi|_g^2 d\mu\Big)^{1/2}.
\end{equation}
We will show that this energy functional satisfies the bounds in Proposition ~\ref{En-prop}.

\medskip 

\begin{proof}[a) Proof of the energy equivalence relation \eqref{eq-E}]

\ 

The relation \eqref{eq-E} for $k\leq s$ with some $s>\frac{d}{2}$ is already a consequence of \eqref{equ-rlt}. We should be more accurate here, we get a better range from  \eqref{equ-rlt}.

It remains to to prove \eqref{eq-E} for $k>s$.
Our starting point is the higher regularity bounds for the elliptic system \eqref{ell-syst}, which were proved in \cite[Section 7.6]{HT}, as follows:
\begin{align}     \label{ell-highReg}
        \|(\la,\SS)\|_{H^\si\times\mathcal H^\si} \lesssim \| \psi\|_{H^\si}, \quad \si\geq s.
\end{align}
This implies in particular that 
\begin{align}     \label{ell-highReg+}
        \|(\Gamma,A)\|_{H^{\si+1}} \lesssim \| \psi\|_{H^\si}, \quad \si\geq s.
\end{align}

By the expression \eqref{nabT=dT}, Sobolev embeddings and \eqref{ell-highReg+} we have
    \begin{align*}
        \|\psi\|_{\mathsf H^k}&\lesssim \|\psi\|_{H^k}+\sum_{1\leq n\leq k}\ \sum_{l_1+\cdots+l_{n+1}\leq  k-n}\|\d_x^{l_1}(\Ga+iA)\cdots \d_x^{l_n}(\Ga+iA)\cdot \d_x^{l_{n+1}}\psi\|_{L^2}\\
        &\lesssim \|\psi\|_{H^k}+\sum_{1\leq n\leq k}\|\Ga+iA\|_{H^k}\|\Ga+iA\|_{H^s}^{n-1}\|\psi\|_{H^s}
        +\sum_{1\leq n\leq k}\|\Ga+iA\|_{H^s}^n\|\psi\|_{H^k}\\
        &\lesssim \|\psi\|_{H^k}+\sum_{1\leq n\leq k}\|\psi\|_{H^k}\|\psi\|_{H^s}^n\\
        &\lesssim \|\psi\|_{H^k}.
    \end{align*}
    Conversely, by \eqref{nabT=dT} we also have
	\begin{align*}
	\|\psi\|_{H^k}\lesssim &\| \psi\|_{\mathsf H^k}+\sum_{1\leq n\leq k}\ \sum_{l_1+\cdots+l_{n+1}\leq  k-n}\|\d^{l_1}(\Ga+iA)\cdots \d^{l_n}(\Ga+iA)\cdot \d^{l_{n+1}}\psi\|_{L^2}\\
	\lesssim &\|\psi\|_{\mathsf H^k}+\ep_0\| \psi\|_{H^k}.
	\end{align*}
	Thus we obtain the equivalence relation $\|\psi\|_{\mathsf H^k}\approx \|\psi\|_{H^k}$.
	
	In the same way as the above, we also have the equivalence $\|\la\|_{\mathsf H^k}\approx \|\la\|_{H^k}$. By \eqref{ell-highReg}, \eqref{ell-lin} and the bound 
	$$\|\psi\|_{H^k}=\|g^{\al\be} \la_{\al\be}\|_{H^k}\lesssim \|\la\|_{H^k}+\ep_0\|\la\|_{H^k}\lesssim \|\la\|_{H^k},$$
	we also have the equivalence $\|\psi\|_{H^k}\approx\|\la\|_{H^k}$. Hence, the desired equivalence relations \eqref{eq-E} are obtained.
\end{proof}

\medskip 

\begin{proof}[b) Proof of the energy estimates \eqref{Energy}]

\ 

\emph{Step 1: Prove that the time derivative of $E^k$ has the form} 
\begin{equation}  \label{Energy-old}
  \frac{d}{dt} E^k(\psi) = \sum_{\sum |\alpha_j|\leq 2k} \int 
\Re  \prod_{j=1}^{J = 4} \nabla^{A,\alpha_j} \la \,  d\mu .
\end{equation}  
with coefficients depending on the metric $g$ so that each of the terms in the above integrand is covariant.

We recall the Schr\"odinger equation \eqref{mdf-Shr-sys-2} first
\begin{equation}\label{Sch-re}
    i(\d_t^B-V^\ga\nab^A_\ga)\psi+\De_g^A\psi
    =\ -i\la_{\si}^{\ga}\Im(\psi\bar{\la}^{\si}_{\ga}).
\end{equation}
Since the energy \eqref{def-energy} does not depend on the choice of gauge, we can choose the advection field $V=0$. Then the volume form $d\mu=\sqrt{\det g}$ is preserved along time $t$.

Applying $\frac{d}{dt}$ to $\| | \nab^{A,k}\psi|_g (t)\|_{L^2(d\mu)}^2$, by \eqref{g_dt} we obtain
	\begin{align}
	& \frac{d}{dt}\int_\Si |\nab^{A,k}\psi|_g^2 d\mu\nonumber\\
    &=\int_{\mathbb R^d} 2\Re g( \nabla^B_t\nabla^{A,k}\psi, \overline{\nabla^{A,k}\psi}) \ d\mu+ (\nabla_tg) (\nabla^{A,k}{\psi}, \overline{\nabla^{A,k}\psi}) \ d\mu\nonumber\\
    &=\int_{\mathbb R^d} 2\Re g( \nabla^B_t\nabla^{A,k}\psi, \overline{\nabla^{A,k}\psi}) +2G (\nabla^{A,k}\la, \overline{\nabla^{A,k}\psi}) \ d\mu. \label{dt-Energy}
\end{align}
	By the equalities \eqref{dt-Ga} and \eqref{Cpt-A&B} with $V= 0$, we have
	\begin{align*}
	[\partial^B_t,\nabla^{A,k}]\psi=\sum_{l_1+l_2+l_3=k} \nab^{A,l_1}\la*\nab^{A,l_2}\psi*\nab^{A,l_2}\psi.
	\end{align*}
Moreover, note that by Gauss equation, the curvature tensor ${\bf R}$ on $\Sigma$ can be
expressed as ${\bf R}= \la*\la$, so the following commutator equality holds
\begin{align}   \label{comm-De}
	[\nabla^{A,k},\Delta^A ]\psi=\sum_{i+j+m=k } \nab^{A,i}\la*\nab^{A,j}\la*\nab^{A,m}\psi.
\end{align}
So by \eqref{Sch-re}, the  first term in the right-hand side of \eqref{dt-Energy} reduces to
\begin{align*}
&\int_{\mathbb R^d} 2\Re g( \nabla^B_t\nabla^{A,k}\psi, \overline{\nabla^{A,k}\psi})\ d\mu\nonumber\\
&=\int_{\mathbb R^d} 2\Re g(\nabla^{A,k}\partial^B_t \psi, \overline{\nabla^{A,k}\psi}) +\sum_{l_1+l_2+l_3=k} \Re  g(\nab^{A,l_1}\la*\nab^{A,l_2}\la*\nab^{A,l_3}\psi, \overline{\nabla^{A,k}\psi})\ d\mu\nonumber\\
&=\int_{\mathbb R^d} 2\Re g(\nabla^{A,k}i\Delta^A\psi, \overline{\nabla^{A,k}\psi})
+\sum_{l_1+l_2+l_3=k} \Re  g(\nab^{A,l_1}\la*\nab^{A,l_2}\la*\nab^{A,l_3}\psi, \overline{\nabla^{A,k}\psi})\ d\mu\nonumber\\
&=\int_{\mathbb R^d} -2\Re i|\nab^{A,k+1}\psi|^2_g
+\sum_{l_1+l_2+l_3=k} \Re  g(\nab^{A,l_1}\la*\nab^{A,l_2}\la*\nab^{A,l_3}\psi, \overline{\nabla^{A,k}\psi})\ d\mu   \nonumber\\
&= \int_{\R^d} \sum_{l_1+l_2+l_3=k} \Re  g(\nab^{A,l_1}\la*\nab^{A,l_2}\la*\nab^{A,l_3}\psi, \overline{\nabla^{A,k}\psi})\ d\mu  
\end{align*}
Hence, we obtain the energy relation \eqref{Energy-old}.  

\medskip 

\emph{Step 2: Proof of energy bound \eqref{Energy}.}

Let us first recall the following interpolation inequality proved by Hamilton \cite[Section 12]{H82}.
\begin{lemma}  
    If T is any tensor and if $1 \leq  i \leq  l-1$, then with a constant $C=C(d,l)$ depending only on dimensions $d$ and $l$, which is independent of the metric $g$ and the connection $\Ga$, we have the estimate
    \begin{equation*}
        \int_{\R^d} |\nab^i T|^{\frac{2l}{i}} \ d\mu\leq C |T|_{L^\infty}^{2(\frac{l}{i}-1)}\int_{\R^d} |\nab^l T|^2 \ d\mu.
    \end{equation*}
\end{lemma}

Then by the interpolation inequality, \eqref{Energy-old}, \eqref{ell-lin} and \eqref{ell-highReg}, for each integer $k$ we have
    \begin{align*}
        \frac{d}{dt}E^k(\psi) &\lesssim \sum_{m\leq k}\sum_{i+j+l\leq m} \|\nab^{A,i}\la\|_{L^{2m/i}}\|\nab^{A,j}\la\|_{L^{2m/j}}\|\nab^{A,l}\psi\|_{L^{2m/l}}\|\nab^{A,m}\psi\|_{L^2}\\&
        \lesssim \|\la\|_{L^\infty}^2 \|\la\|_{\mathsf H^k}^2.
    \end{align*}
Thus we obtain the energy estimates \eqref{Energy}.
\end{proof}

\medskip

Finally, we prove an energy estimate in negative Sobolev spaces $H^{-1}$ for the linearized equation of \eqref{mdf-Shr-sys-2}.

\begin{prop}\label{Diff-est}
Let $d\geq 4$. Under the assumptions \eqref{ini-data} and \eqref{prop-Ass1}, for the linearized equation of \eqref{mdf-Shr-sys-2} we have the bound
\begin{equation}  \label{Diff-est-1}
\|\psi_{lin}\|_{L^\infty_T H^{-1}} \leq C_{lin}\|\psi_{lin}(0)\|_{ H^{-1}}+C_{lin}\| \psi\|_{S[0,T]}^2 \| \psi_{lin}\|_{L^\infty_T H^{-1}}.
\end{equation}
\end{prop}
For clarity, here we note that the linearized equation depends on our gauge choices. The above proposition and its proof below assume we are in the harmonic/Coulomb gauge.

\begin{proof}
Beginning with the $\psi$-equation in \eqref{mdf-Shr-sys-2}, we write the associated linearized equation in the form
    \begin{align*}
        &i(\d^{B}_t-V^{\ga}\nab^A_{\ga})\psi_{lin}+\De^A_g\psi_{lin}\\&=(V^\ga\nab^A_\ga)_{lin}\psi-(\De_g^A)_{lin}\psi+B_{lin}\psi-[i\la_{\si}^{\ga}\Im(\psi\bar{\la}^{\si}_{\ga})]_{lin}
        :=F_{lin},
    \end{align*}
  Here we will treat the source term $F_{lin}$ perturbatively. This allows us to split the proof of \eqref{Diff-est-1} into two parts. Precisely, it suffices to prove the linear bound
    \begin{equation}\label{lin-est-low}
    \|\psi_{lin}\|_{L^\infty_T H^{-1}}\lesssim \|\psi_{lin}(0)\|_{ H^{-1}}+\|F_{lin}\|_{L^1_T  H^{-1}}.  
    \end{equation}
  respectively the source term estimate
  \begin{equation}\label{pert-est-low}
   \|F_{lin}\|_{L^1_T  H^{-1}} \lesssim    \|\psi_{lin}\|_{L^\infty_T H^{-1}}\|\psi\|_{S[0,T]}^2
  \end{equation}
  Together, these two bounds imply the conclusion of the proposition.
  It remains to prove \eqref{lin-est-low} and \eqref{pert-est-low}.

    We first consider the bound \eqref{lin-est-low}, which we prove
    using duality. For this we need the associated adjoint equation,
    which has the form
    \begin{equation}    \label{adEq-v}
    	i(\d_t^B -V^\ga\nab^A_\ga)v+\De_g^A v-i\nab_\ga V^\ga v=\mathcal N,  
    \end{equation}
    The adjoint evolution is considered in the same time interval $[0,T]$, but 
    as a backward Cauchy problem with the initial data at time $T$.
    Then we claim that this evolution is (backward) well-posed in $\mathsf H^1$, and satisfies the the bound
    \begin{align}   \label{claim-v}
        \|v\|_{L^\infty(0,T; \mathsf H^1)}\lesssim \|v(T)\|_{\mathsf H^1}+\|\mathcal N\|_{L^1(0,T; \mathsf H^1)}.
    \end{align}
    Assuming this holds, then from the duality relation
    \[
    \left.\< \psi_{lin},v\> \right|_0^T = \<-i\mathcal N,F_{lin}\>
    \]
    we have the bound
    \begin{align*}
        |\< (\psi_{lin}(T),\psi_{lin}),(v(T),-i\mathcal N)  \>|\lesssim (\|\psi_{lin}(0)\|_{\mathsf H^{-1}}+\|F_{lin}\|_{L^1 \mathsf H^{-1}})(\|v(T)\|_{\mathsf H^1}+\|\mathcal N\|_{L^1\mathsf H^1})
    \end{align*}
    which in turn implies that 
    \[ \|\psi_{lin}\|_{L^\infty_T \mathsf H^{-1}}\lesssim \|\psi_{lin}(0)\|_{\mathsf H^{-1}}+\|F_{lin}\|_{L^1_T \mathsf H^{-1}}.  \]
    Since the metric $g-I_d$ and connection $A$ are small in harmonic/Coulomb gauge, by equivalence \eqref{equ-rlt} and duality we have
    \begin{equation*}
        \|u\|_{\mathsf H^{-1}}=\sup_{\|v\|_{\mathsf H^1}\leq 1}\<u,v\>_{L^2}\lesssim \sup_{\|v\|_{ H^1}\leq 1}\<u,v\>_{L^2}= \| u\|_{H^{-1}}.
    \end{equation*}
    Then the desired bound \eqref{lin-est-low} follows.

    Now we prove the bound \eqref{pert-est-low} for the nonlinear terms $F_{lin}$. This is a consequence of the fixed time bound
    \begin{equation}\label{pert-fixed-time}
    \| F_{lin}\|_{H^{-1}} \lesssim      \|\psi_{lin}\|_{H^{-1}}\| \psi\|_{\mathbf{str}}^2,
    \end{equation}
    which we now prove by successively considering all the terms in $F_{lin}$.
    
    Using Sobolev embeddings and \eqref{linEst} we bound the worst term $(\De_g)_{lin}\psi$ by
    \begin{align*}
        \|(\De_g)_{lin}\psi\|_{H^{-1}}&= \|h^{\al\be}_{lin}\d^2_{\al\be}\psi\|_{H^{-1}}\\
        &\lesssim (\|h_{lin}\|_{L^{\frac{2d}{d-2}}}+\|\d_x h_{lin}\|_{L^2})\| \psi\|_{W^{1,d}}\\
        &\lesssim  \|\la_{lin}\|_{H^{-1}}\| \psi\|_{\mathbf{str}}^2\\
        &\lesssim  \|\psi_{lin}\|_{H^{-1}}\| \psi\|_{\mathbf{str}}^2.
    \end{align*}
   For the term $A^\al_{lin} \d_\al \psi$, by \eqref{linEst} we have
   \begin{align}\label{Adpsi}
       \|A^\al_{lin} \d_\al \psi\|_{H^{-1}}\lesssim \|A_{lin}\|_{L^2}\|\d_x\psi\|_{L^d}
       \lesssim \|\psi_{lin}\|_{H^{-1}}\|\psi\|_{\mathbf{str}}^2.
   \end{align}
    Similarly, by \eqref{linEst} we also have
    \begin{equation} \label{A2psi}
    \begin{aligned}
       &\|(\nab_\al A^\al)_{lin}  \psi\|_{H^{-1}}+\|(A^\al A_\al)_{lin}  \psi\|_{H^{-1}}+\|(\la^3)_{lin} \|_{H^{-1}}\\
       &\lesssim (\|A_{lin}\|_{L^{2}}+\|(A_\al A^\al)_{lin}\|_{L^{2}})\|\psi\|_{\mathbf{str}}+\|\la_{lin}\|_{H^{-1}}\|\la\|_{\mathbf{str}}^2\\
       &\lesssim \|\psi_{lin}\|_{H^{-1}}\|\psi\|_{\mathbf{str}}^2.
   \end{aligned}
   \end{equation}
    and 
    \begin{align*}
       \|B_{lin}  \psi\|_{H^{-1}}\lesssim \|B_{lin}\|_{H^{-1}}(\|\psi\|_{L^\infty}+\|\psi\|_{W^{1,d}})
       \lesssim \|\psi_{lin}\|_{H^{-1}}\|\psi\|_{\mathbf{str}}^2.
   \end{align*}
    For the term $(V^\ga\nab^A_\ga)_{lin}\psi$, by the same argument as \eqref{Adpsi} and \eqref{A2psi} and the estimate \eqref{linEst} we bound it  by
    \begin{align*}
    \|(V^\ga\nab^A_\ga)_{lin}\psi\|_{H^{-1}}&\lesssim \|V_{lin}\d_x\psi\|_{H^{-1}}+\|(VA)_{lin}\psi\|_{H^{-1}}\\
    &\lesssim \|V_{lin}\|_{L^2}\|\d_x\psi\|_{L^d}+\|(VA)_{lin}\|_{L^2}\|\psi\|_{\mathbf{str}}\\
    &\lesssim \|\psi_{lin}\|_{H^{-1}}\|\psi\|_{\mathbf{str}}^2
    \end{align*}
    This concludes the proof of \eqref{pert-fixed-time} and thus of 
    \eqref{pert-est-low}.
    
    \medskip

    Finally, we turn to the proof of the claim \eqref{claim-v}. Since this proof is more complicated than that of \propref{En-prop}, we provide the full details. By \eqref{adEq-v} and integration by parts, we have the basic energy estimate
    \begin{align*}
    	\frac{d}{dt}\|v\|_{\mathsf L^2}^2&= \int 2\Re\<\d^B_t v,v\>+|v|^2 \nab_\al V^\al\ d\mu\\
    	&=\int 2\Re\<(V^\ga \nab^A_\ga v+i\De^A_g v+\nab_\ga V^\ga v-i\mathcal N),v\>+|v|^2\nab_\al V^\al\ d\mu \\
    	&=\int 2\nab_\ga V^\ga |v|^2 -2\Re\<i\mathcal N,v\>\ d\mu\\
    	&\lesssim \|v\|_{\mathsf L^2}^2 \|\nab^\ga V_\ga\|_{L^\infty}+\|\mathcal N\|_{\mathsf L^2}\|v\|_{\mathsf L^2}\\
    	&\lesssim \|v\|_{\mathsf L^2}^2 \|\la\|_{W^{\si_d,r_d}}+\|\mathcal N\|_{\mathsf L^2}\|v\|_{\mathsf L^2}.
    \end{align*}

    We then derive an energy estimate for $\nab^A v$ in $ \mathsf L^2$. By \eqref{adEq-v} and integration by parts we have
    \begin{align*}
        \frac{d}{dt}\int |\nab^A v|^2\ d\mu&=\int 2\Re \<[\d^B_t, \nab^A_\al] v+\nab^A_\al \d^B_t v,\nab^{A,\al}v\>\\
        &+\Re\<\nab^A_\al v,-2G^{\al\be}\nab^A_\be v\>+|\nab^A v|^2\nab_\al V^\ga\ d\mu \\
        &=\int 2\Re \<[\d^B_t, \nab^A_\al] v,\nab^{A,\al}v\>+\Re\<\nab^A_\al v,-2G^{\al\be}\nab^A_\be v\>+|\nab^A v|^2\nab_\al V^\ga\ d\mu \\
        &+\int 2\Re\<\nab^A_\al (V^\ga \nab^A_\ga v+i\De^A_g v+\nab_\ga V^\ga v-i\mathcal N),\nab^{A,\al}v\>\ d\mu\\
        &=\int 2\Re \<[\d^B_t-i\De^A_g, \nab^A_\al] v,\nab^{A,\al}v\>\\
        &+\Re\<\nab^A_\al v,-2G^{\al\be}\nab^A_\be v\>+2|\nab^A v|^2\nab_\al V^\ga\ d\mu \\
        &+\int 2\Re\<V^\ga [\nab^A_\al, \nab^A_\ga]v+\nab_\al V^\ga \nab^A_\ga v+\nab_\al\nab_\ga V^\ga v-i\nab^A_\al\mathcal N),\nab^{A,\al}v\>\ d\mu.
    \end{align*}
This implies
\begin{align*}
	\frac{d}{dt}\int |\nab^A v|^2\ d\mu&\lesssim \|[\d^B_t-i\De^A_g, \nab^A_\al] v\|_{\mathsf L^2}\|\nab^A v\|_{\mathsf L^2}+\|V^\ga[\nab^A_\al,\nab^A_\ga]v\|_{\mathsf L^2}\|\nab^A v\|_{\mathsf L^2}\\
	&+\|v\|_{\mathsf H^1}^2 (\|\la\|_{L^\infty}^2+\|\nab V\|_{L^\infty})+\|\nab_\al\nab_\ga V^\ga\|_{L^d}\|v\|_{L^{2d/(d-2)}}\|v\|_{\mathsf H^1}\\
	&+\|\mathcal N\|_{H^1}\|v\|_{\mathsf H^1}.
\end{align*}
We use \eqref{Cpt-A&B} and \eqref{comm-De} to bound the first term by 
\begin{align*}
	\|[\d^B_t-i\De^A_g, \nab^A_\al] v\|_{\mathsf L^2}\|\nab^A v\|_{\mathsf L^2}&\lesssim \|(\la*\nab^A\psi+\la^2 V)v\|_{L^2}\|v\|_{H^1}\\
	&+\|\nab^A\la*\la*v+\la*\la*\nab v\|_{L^2}\|v\|_{H^1}\\
	&\lesssim (\|\la*\nab^A\la\|_{L^d}+\|\la^2V\|_{L^d}) \|v\|_{L^{2d/(d-2)}}\|v\|_{H^1}\\
	&+\|\la\|_{L^\infty}^2\|v\|_{H^1}^2\\
	&\lesssim \|\la\|_{\mathbf{str}}^2(1+\|\nab V\|_{H^{s-2}})\|v\|_{H^1}
\end{align*}
For the second term we have
\begin{align*}
	\|V^\ga[\nab^A_\al,\nab^A_\ga]v\|_{\mathsf L^2}\|\nab^A v\|_{\mathsf L^2}=\|iV^\ga \Im(\la_{\al\mu}\bar{\la}^\mu_\ga)v\|_{L^2}\|v\|_{H^1}\lesssim \|\la\|_{\mathbf{str}}^2\|\nab V\|_{H^{s-2}}\|v\|_{H^1}.
\end{align*}
For the third and forth terms, by Sobolev embeddings and \eqref{L1t-V} we have
\begin{align*}
	\|v\|_{\mathsf H^1}^2 (\|\la\|_{L^\infty}^2+\|\nab V\|_{L^\infty})+\|\nab_\al\nab_\ga V^\ga\|_{L^d}\|v\|_{L^{2d/(d-2)}}\|v\|_{\mathsf H^1}
	\lesssim \|v\|_{H^1}\|\la\|_{\mathbf{str}}^2.
\end{align*}

Hence, we conclude that
\begin{align*}
	\|v\|_{L^\infty_T \mathsf H^1}&\lesssim \|v_0\|_{\mathsf H^1}+\|\la\|_{L^2_T\mathbf{str}}^2\|v\|_{L^\infty_T \mathsf H^1} +\|\mathcal N\|_{L^1_T\mathsf H^1}\\
	&\lesssim \|v_0\|_{\mathsf H^1}+\|\la\|_{S[0,T]}^2\|v\|_{L^\infty_T \mathsf H^1} +\|\mathcal N\|_{L^1_T\mathsf H^1}
\end{align*}
By the assumption \eqref{prop-Ass1} and H\"older's inequality, this yields the claim \eqref{claim-v}. This completes the proof of the proposition.
\end{proof}

\bigskip

\section{Strichartz estimates}   \label{Sec-Strichartz}

Here we consider the Schr\"odinger equation \eqref{mdf-Shr-sys-2}, and prove the Strichartz bounds in \propref{impbd-1}. First, we introduce the endpoint Strichartz estimates of Keel-Tao \cite{KT} and the inhomogeneous Strichartz estimates developed by \cite{F05, Vilela, Koh, KohSeo}. Then we use these to bound the linear and nonlinear part, respectively.

We begin with the homogeneous Strichartz estimates obtained by Keel-Tao \cite{KT}:
\begin{equation}         \label{KT}
    \| e^{it\De} f\|_{L^q L^r}\lesssim \| f\|_{L^2},
\end{equation}
where $(q,r)$ is Schr\"odinger-admissible pair, that is, $\frac{2}{q}+\frac{d}{r}=\frac{d}{2}$, $2\leq q,\ r\leq \infty$, $(q,r,d)\neq (2,\infty ,2)$. Here we will use the endpoint pair $(q,r)=(2,\frac{2d}{d-2})$.

Next, we state the inhomogenous Strichartz estimates, which summarize several known results, see \cite{F05, Vilela, Koh, KohSeo}. 
\begin{defn}
We say that the pair $(q,r)$ is Schr\"odinger-acceptable if 
\begin{equation*}
    1\leq q<\infty,2\leq r\leq \infty,\quad \frac{1}{q}<\frac{d}{2}(1-\frac{2}{r}),\quad \text{or }(q,r)=(\infty,2).
\end{equation*}
\end{defn}

\begin{thm}[Inhomogeneous Strichartz estimates]
Let $d\geq 3$ and $p'$ be the duality of $p$ with $\frac{1}{p}+\frac{1}{p'}=1$. Assume that the pairs $(q,r)$ and $(\tilde q, \tilde r)$ are Schr\"odinger-acceptable pairs, and satisfy the condition
\begin{equation*}
    \frac{1}{q}+\frac{1}{\tilde q}=\frac{d}{2}(1-\frac{1}{r}-\frac{1}{\tilde r}).
\end{equation*}
In addition, assume one of the following:
\begin{itemize}
    \item [i)] non-sharp case:
        \begin{gather*}
            \frac{1}{q}+\frac{1}{\tilde q}<1,\qquad 
            \frac{d-2}{d}\ \frac{1}{r}\leq \frac{1}{\tilde r}\leq \frac{d}{d-2}\ \frac{1}{r},\qquad \frac{1}{r},\ \frac{1}{\tilde r}\leq \frac{1}{2};
        \end{gather*}
    \item [ii)]sharp case:
        \begin{gather*}
            \frac{1}{q}+\frac{1}{\tilde q}=1,\qquad
            \frac{d-2}{d}\ \frac{1}{r}<\frac{1}{\tilde r}< \frac{d}{d-2}\ \frac{1}{r},\qquad
            \frac{1}{r}\leq \frac{1}{q},\quad \frac{1}{\tilde r}\leq \frac{1}{\tilde q};
        \end{gather*}
    \item [iii)] endpoint cases when $d\geq 3$:
    \begin{gather*}
            \frac{1}{q}+\frac{1}{\tilde q}=1,\qquad
            \frac{1}{r}=\frac{d}{d-2}\ \frac{1}{\tilde r}\quad \text{or}\quad \frac{d-2}{d}\ \frac{1}{\tilde r},\qquad
            \frac{1}{r}\leq \frac{1}{q},\quad \frac{1}{\tilde r}\leq \frac{1}{\tilde q}.
        \end{gather*}
\end{itemize}
Then the following estimate  holds 
\begin{equation}        \label{Strichartz}
    \lVert\int_0^t  e^{i(t-s)\De} F(s)ds \rVert_{L^qL^r}\lesssim \|F\|_{L^{\tilde q '}L^{\tilde r'}}.
\end{equation}
\end{thm}

\bigskip 

We now aim to prove  the space-time bound $S[0,T]$ for $\psi$
in Proposition~\ref{impbd-1} by combining the above Strichartz estimates with the elliptic estimates in section \ref{Ellip-sec}.

\begin{proof}[Proof of \propref{impbd-1}]
	By Duhamel's principle, the solution $\psi$ of \eqref{mdf-Shr-sys-2} can be expressed by
	\begin{align*}
	\psi(t)=e^{it\De}\psi_0+\int_0^t e^{i(t-s)\De} \mathcal N(s) ds,
	\end{align*}
	where 
	\begin{align*}
	    \mathcal N:=h\d_x^2\psi+(V+A)\d_x\psi+(B+A^2+VA+\la^2)\psi.
	\end{align*}
	Using Sobolev embeddings, the bound \eqref{KT} with $(q,r)=(2,\frac{2d}{d-2})$ 
	and the estimates \eqref{Strichartz} with pairs $(q,r)=(2,r_d)$, $ (\tilde q,\tilde r)=(2,\frac{2(d-1)}{(d-2)})$, we have
	\begin{align*}
	\|\psi\|_{L^2_T W^{\si_d,r_d}}
	\lesssim \|\psi_0\|_{H^{\si_d+\frac{d-2}{2(d-1)}}}+\| \mathcal N\|_{L^2_T W^{\si_d,\tilde r'}}.
	\end{align*}
	
	It remains to successively estimate the  terms in $\mathcal N$.
	For the first term $h\d_x^2 \psi$, by \eqref{L2t-SS} and \eqref{ell-bd} we have
	\begin{align*}
	\| h\d_x^2\psi\|_{L^2_T W^{\si_d,\tilde r'}}\lesssim  \| h\|_{L_T^2 L^{2(d-1)}}\|\psi\|_{L_T^\infty H^{\si_d+2}}+\|\nab^2 h\|_{L^\infty_T H^{\si_d+\ep}}\|\psi\|_{L^2_T L^{2(d-1)}}\lesssim \ep_0 \|\psi\|_{L^2_T W^{\si_d,r_d}}.
	\end{align*}
	For the second term $(V+A)\d_x \psi$, by \eqref{L2t-SS} and \eqref{ell-bd} we have
	\begin{align*}
	\| (V+A)\d_x \psi\|_{L^2_T W^{\si_d,\tilde r'}}\lesssim  (\| \d_x V\|_{L_T^2H^{s_d}}+\| \d_x A\|_{L_T^2H^{s_d}})\|\psi\|_{L_T^\infty H^{s_d}}
	\lesssim \ep_0 \|\psi\|_{L^2_T W^{\si_d,r_d}}.
	\end{align*}
	Finally, for the other terms, by \eqref{L2t-la} we also have
	\begin{align*}
	&\|(B+A^2+VA+\la^2)\psi\|_{L^2_T W^{\si_d,\tilde r'}}\\
	&\lesssim   (\|\d_x B\|_{L_T^2H^{s_d-1}}+\| \d_x A\|_{L_T^2H^{s_d}}+\|\d_x V\|_{L^2_T H^{s_d}})(1+\|\SS\|_{L_t^\infty \HH^{s_d}})\|\psi\|_{L_T^\infty H^{s_d}}\\
	&\quad +\|\la\|_{L^2_T W^{\si_d,r_d}}\|\la\|_{L_T^\infty H^{s_d}}\|\psi\|_{L_T^\infty H^{s_d}}\\
	&\lesssim \ep_0 \|\psi\|_{L^2_T W^{\si_d,r_d}}.
	\end{align*}
	This concludes the proof of the bound \eqref{psi-L2t} for $\psi\in L^2_T W^{\si_d,r_d}$.
	
	In order to obtain the bound for $\psi\in L^2_TW^{1,4}$, by \eqref{KT} and \eqref{Strichartz} we have
	\begin{align*}
	\|\psi\|_{L^2_T W^{1,4}}
	\lesssim \|\psi_0\|_{H^{1}}+\| \mathcal N\|_{L^2_T W^{1,\frac{4}{3}}}.
	\end{align*}
	Using \eqref{L2t-hW14}, \eqref{L2t-SS}, \eqref{L2t-laW14} and \eqref{prop-Ass1} we bound the nonlinear terms by 
	\begin{align*}
	\| \mathcal N\|_{L^2_T W^{1,4}}&\lesssim  \| h\|_{L_T^2 W^{1,4}}\|\psi\|_{L_T^\infty H^{3}}\\
	&\quad +(\|\d_x B\|_{L_T^2L^2}+\| \d_x A\|_{L_T^2L^2}+\|\d_x V\|_{L^2_T L^2})(1+\|\SS\|_{L_t^\infty \HH^{s_d}})\|\psi\|_{L_T^\infty H^{2}}\\
	&\quad +\|\la\|_{L^2_T W^{1,4}}\|\la\|_{L_T^\infty H^{s_d}}\|\psi\|_{L_T^\infty H^{1}}\\
	&\lesssim (C\|\psi_0\|_{H^{3}})^2 \|\psi\|_{L^2_TW^{1,4}}.
	\end{align*}
	These imply the bound \eqref{psi-L2t} for $\psi\in L^2_T W^{1,4}$ in dimension$d=4$.
\end{proof}

\bigskip

\section{Rough solutions and scattering}   \label{sec-proof}
In this section, we use elliptic estimates in section \ref{Ellip-sec}, energy estimates \eqref{E-psi}, \eqref{Diff-est-1} and Strichartz estimates in \propref{impbd-1} to prove the improved energy bounds \eqref{prop-result} for $\psi$ in fractional Sobolev spaces. This closes the proof of \propref{bootstrap-prop}. As a byproduct, we also obtain the scattering property \eqref{Scatter}.

Here we start with an equivalent definition of $H^s$. Since in the Hilbertian case all interpolation methods yield the same result, for the $H^s$ norm we will use a characterization which is akin to a Littlewood-Paley decomposition, or to a discretization of the $J$ method of interpolation. Precisely, we
have
\begin{lemma}
    Let $0\leq s\leq N$. Then $H^s$ can be defined as the space of distributions $u$ which admit a representation
\[  u=\sum_{j=0}^\infty u_j  \]
with the property that the following norm is finite:
\[ \tri (u_j)\tri^2_s =\sum_{j=0}^\infty 2^{2j(s+1)} \|u_j\|_{H^{-1}}^2 +2^{2j(s-N)}\|u_j\|_{H^N}^2    \]
and with equivalent norm defined as
\begin{equation}\label{Eq-Hs}
    \|u\|_{H^s}^2=\inf \tri (u_j)\tri ^2_s,
\end{equation}
where the infimum is taken with respect to all representations as above.
\end{lemma}

\subsection{Regularized data}
Consider an initial data $\psi_0\in H^{s_d}$ small, and let $\{c_k\}_{k\geq 0}$ be a sharp frequency
envelope for $\psi_0$ in $H^{s_d}$. For $\psi_0$ we consider a family of regularizations at frequencies $\lesssim 2^k$, i.e.
\[\psi^{(k)}_0:=S_{\leq k} \psi_0\ \in H^\infty :=\cap_{j=0}^\infty H^j,\]
where $k$ is a dyadic frequency parameter. This parameter can be taken either discrete or continuous,
depending on whether we have access to difference bounds or only to the linearized equation. Suppose we
work with differences. Then the family $\psi^{(k)}_0$ can be taken to have similar properties to Littlewood-Paley
truncations:
\begin{itemize}
    \item [i)] Uniform bounds:
    \begin{equation*}
        \|S_j \psi^{(k)}_0\|_{H^{s_d}}\lesssim c_j.
    \end{equation*}
    
    \item [ii)] High frequency bounds: for $\si>\de$,
    \begin{equation}      \label{high-ini}
        \| \psi^{(k)}_0\|_{H^{s_d+\si}}\lesssim 2^{\si k}c_k.
    \end{equation}
    
    \item [iii)] Difference bounds: 
    \begin{equation}      \label{Diff-L2ini}
        \| \psi^{(k+1)}_0-\psi^{(k)}_0\|_{H^{-1}}\lesssim 2^{-(s_d+1) k}c_k.
    \end{equation}
    
    \item [iv)] Limit as $k\rightarrow \infty$: 
    \begin{equation*}
        \psi_0=\lim_{k\rightarrow\infty} \psi_0^{(k)}\quad \text{in } H^{s_d}.
    \end{equation*}
\end{itemize}
Correspondingly, we obtain a family of smooth solutions $\psi^{(k)}$.

\subsection{Uniform bounds}
Corresponding to the above family of regularized data, we obtain a family of smooth
solutions $\psi^{(k)}$ on $[0,T_n]$ for $T_n>1$ by Theorem \ref{small-MSS-LWP}. For this we can use the energy estimates \eqref{E-psi} to propagate Sobolev regularity for solutions as well as difference bounds as in Proposition \ref{Diff-est}. 
Using induction we will prove that the solutions $\psi^{(k)}$ are global as follows:

(i) We prove that the solution $\psi^{(0)}$ is global. By local well-posedness in \thmref{small-MSS-LWP}, let $T_0$ be
\[ T_0=\sup_{T} \Big\{T: \|\psi^{(0)}\|_{S[0,T]}+\|\psi^{(0)}\|_{L^\infty_T H^{s_d}}\leq C_0 \|\psi_0\|_{H^{s_d}}     \Big\}.         \]
Then on the interval $[0,T_0]$, by \propref{impbd-1} we have
\[\|\psi^{(0)}\|_{S[0,T]}\leq 2C_2\|\psi^{(0)}_0\|_{H^{s_d}}\leq 2C_2\|\psi_0\|_{H^{s_d}} .\]
Using \eqref{Diff-est-1} and \eqref{E-psi} we have
\begin{align*}
    \|\psi^{(0)}\|_{L^\infty_{T_0} H^{s_d}}&\leq \|\psi^{(0)}\|_{L^\infty_{T_0}H^{-1}}+\|\psi^{(0)}\|_{L^\infty_{T_0}H^N}\\
    &\leq (1-C_{lin}C_0^2\ep_0^2)^{-1}C_{lin}\|\psi_0\|_{H^{-1}} +C_1 e^{C_E C_0^2 \ep_0^2 }\|\psi_0\|_{H^{s_d}}\\
    &\leq \frac{C_0}{2}\|\psi_0\|_{H^{s_d}}.
\end{align*}
Here we set 
\begin{equation}  \label{C0}
    C_0> 4C_1+4C_2+4C_{lin},
\end{equation}
and choose $\ep_0$ to be sufficiently small such that 
\begin{equation}   \label{ep0}
    (C_E+C_{lin})C_0^2\ep_0^2 <\frac{1}{4}.
\end{equation}  
This implies that the solution can be extended, and thus the lifespan is $T_0=\infty$.

(ii) We prove that the the solutions $\psi^{(k)}$ for any $k$ are global. By local well-posedness in \thmref{small-MSS-LWP}, let $T_k$ be
\[ T_k=\sup_{T} \Big\{T: \|\psi^{(k)}\|_{S[0,T]}+\|\psi^{(k)}\|_{L^\infty_T H^{s_d}}\leq C_0 \|\psi_{0}\|_{H^{s_d}}     \Big\}.         \]
Then on the interval $[0,T_k]$, by \propref{impbd-1} we have the improved Strichartz estimates in $S[0,T]$. We then prove the improved energy estimates for $\psi^{(k)}$.

By (i) we assume that $\psi^{(l)}$ for $l\leq k-1$ are global. Then we have two properties as follow:
\begin{itemize}
    \item [a)] High frequency bounds:
    \begin{equation}    \label{Hi-bd}
        \|\psi^{(l)}\|_{C[0,T_k;H^{N_1}]}\leq C_1 e^{C_EC_0\ep_0} 2^{(N_1-s_d) l}c_l,\quad     0\leq l\leq k,\ s_d<N_1\in \N.
    \end{equation}
    
    \item [b)] Difference bounds:
    \begin{equation}\label{Diff-bd}
        \|\psi^{(l+1)}-\psi^{(l)}\|_{C[0,T_k;H^{-1}]}\leq 2C_{lin}2^{-(s_d+1) l}c_l,\quad     0\leq l\leq k-1.
    \end{equation}
\end{itemize}
The first bound is obtained from \eqref{E-psi} and \eqref{high-ini}. The second bound \eqref{Diff-bd} is obtained by \propref{Diff-est}, \propref{impbd-1} and \eqref{Diff-L2ini}. Indeed, 
\begin{align*}
    &\|\psi^{(l+1)}-\psi^{(l)}\|_{C[0,T_k;H^{-1}]}\\
    &\leq C_{lin}\|\psi^{(l+1)}_0-\psi^{(l)}_0\|_{H^{-1}}
    +C_{lin}\| (\psi^{(l+1)},\psi^{(l)})\|^2_{S[0,T_k]}\|\psi^{(l+1)}-\psi^{(l)}\|_{C[0,T_k;H^{-1}]}\\
    &\leq C_{lin} 2^{-(s+1)l}c_l+2C_{lin}(C_0 \ep_0)^2 \|\psi^{(l+1)}-\psi^{(l)}\|_{C[0,T_k;H^{-1}]}.
\end{align*}

Interpolating the two estimates \eqref{Hi-bd} and \eqref{Diff-bd}, we obtain
\begin{equation}       \label{EN-est}
   \|\psi^{(l+1)}-\psi^{(l)}\|_{C[0,T_k;H^N]}\leq \max\{ 2C_{lin},C_1 e^{C_EC_0^2\ep_0^2}  \} 2^{-(s_d-N)l}c_l, \quad   0<N<N_1.
\end{equation}
We use these bounds to establish uniform frequency envelope bounds for $\psi^{(k)}$,
\begin{align*}
    \|\psi^{(k)}\|_{C[0,T_k;H^{s_d}]}^2&\leq  \|\psi^{(0)}\|_{C[0,T_k;H^{-1}]}^2+ \sum_{l=1}^{k-1} 2^{2(s_d+1)l} \|\psi^{(l+1)}-\psi^{(l)}\|_{C[0,T_k;H^{-1}]} ^2\\&
    \quad + \|\psi^{(0)}\|_{C[0,T_k;H^N]}^2+ \sum_{l=1}^{k-1} 2^{2l(s_d-N)} \|\psi^{(l+1)}-\psi^{(l)}\|_{C[0,T_k;H^N]} ^2\\&
    \leq \sum_{l=0}^{k-1} (2C_{lin}c_l)^2 +\sum_{l=0}^{k-1} (\max\{ 2C_{lin},C_1 e^{C_EC_0^2\ep_0^2}  \})^2 c_l^2\\
    &\leq \sum_{l=0}^{k-1} \Big(\frac{C_0}{2} c_l\Big)^2\leq \Big(\frac{C_0}{2}\|\psi_{0}\|_{H^{s_d}}\Big)^2.
\end{align*}
Here $C_0$ and $\ep_0$ are chose as \eqref{C0}, \eqref{ep0} respectively.
This implies that the solutions $\psi^{(k)}$ are also global.

\medskip 

Now consider the convergence of solutions $\psi^{(k)}$ in $C(\R;H^{s_d})$ as $k\rightarrow\infty$. From the difference bounds \eqref{Diff-bd} we obtain convergence in $H^{-1}$ to a limit $\psi\in C[0,\infty;H^{-1}]$, with
\begin{align*}
    \|\psi-\psi^{(k)}\|_{C(\R;H^{-1})}\leq \sum_{l=k}^\infty\|\psi^{(l+1)}-\psi^{(l)}\|_{C(\R;H^{-1})}\leq \sum_{l=k}^\infty 2^{-(s_d+1)l}c_l\lesssim 2^{-(s+1)k}.
\end{align*}
On the other hand, expanding the difference as a telescopic sum, where, in view of the above bounds \eqref{Hi-bd} and \eqref{Diff-bd}, each summand is essentially concentrated at frequency $2^l$, with $H^{s_d}$ size $c_l$ and exponentially decreasing tails. By the equivalent norm \eqref{Eq-Hs}, \eqref{Diff-bd} and \eqref{EN-est} we have
\begin{align*}
    \|\psi-\psi^{(k)}\|_{C(\R;H^{s_d})}^2&\leq \sum_{l=k}^\infty 2^{2(s_d+1)l}\|\psi^{(l+1)}-\psi^{(l)}\|_{C(\R;H^{-1})}^2\\
    &\quad + \sum_{l=k}^\infty2^{2(s_d-N)l}\|\psi^{(l+1)}-\psi^{(l)}\|_{C(\R;H^N)}^2\\&
    \lesssim \sum_{l=k}^\infty c_l^2=c_{\geq k},
\end{align*}
so we also have convergence in $C(\R,H^{s_d})$. 

Hence, we obtain the solution $\psi$ as the limit of solutions $\psi^{(k)}$, and have the bound
\[\|\psi\|_{C(\R;H^{s_d})}\leq \lim_{k\rightarrow\infty} \|\psi^{(k)}\|_{C(\R;H^{s_d})}\leq \frac{C_0}{2}\|\psi_0\|_{H^{s_d}}.\]
This gives the improved energy bound in \propref{bootstrap-prop}. The first improved bound in \eqref{prop-result} for $\psi\in S[0,T]$ is obtained by \propref{impbd-1}, \eqref{C0} and \eqref{ep0}. Hence, we complete the proof of \propref{bootstrap-prop}.

\bigskip

Finally, we prove that scattering holds.
\begin{prop}[Scattering]
	Let $s_d$ be as in \eqref{Reg-index}. There exist $\psi_{\pm}\in H^{s_d-2}$ such that
	\begin{equation} \label{Scatter-re}
	\lim_{t\rightarrow\pm\infty} \|\psi-e^{it\De}\psi_{\pm}\|_{H^{s_d-2}}=0.
	\end{equation}
\end{prop}
\begin{proof}
	It is standard to deduce \eqref{Scatter-re} from
	\[  \|(V_{\alpha}-2A_{\alpha})\nabla^{\alpha}\psi\|_{L_t^2W^{s_d-2,r_d}}+ \|\mathcal{N}\|_{L_t^2W^{s_d-2,r_d}}\lesssim 1,   \]
	which has been proved in Proposition \ref{impbd-1}. So our lemma follows.
\end{proof}

\bigskip

\section*{Acknowledgments}
J. Huang was partially supported by China Postdoctoral Science Foundation Grant 2021M690223. Z. Li was supported by  NSF-China Grant-1200010237.
D. Tataru was supported by NSF grant DMS-2054975 as well as by a Simons Investigator grant from the Simons
Foundation.

\end{document}